Title: Remarks on Semimodules
Author1: Bodo Pareigis (University of Munich, Germany)
Author2: Helmut Rohrl (UCSD, USA)
Comments: 35 pages, PdfLaTeX with amsart.sty
 Abstract:  This is a study of universal problems for semimodules, in particular coequalizers, coproducts, and tensor products. Furthermore the structure theory of semiideals of the semiring of natural numbers is extended.

\documentclass[12pt]{amsart}

\makeatletter
 
\def\diagram{\m@th\leftwidth=\z@ \rightwidth=\z@ \topheight=\z@
\botheight=\z@ \setbox\@picbox\hbox\bgroup}
 
\def\enddiagram{\egroup\wd\@picbox\rightwidth\unitlength
\ht\@picbox\topheight\unitlength \dp\@picbox\botheight\unitlength
\hskip\leftwidth\unitlength\box\@picbox}
 
\def\bfig{\begin{diagram}}
\def\efig{\end{diagram}}
\newcount\wideness \newcount\leftwidth \newcount\rightwidth
\newcount\highness \newcount\topheight \newcount\botheight
 
\def\ratchet#1#2{\ifnum#1<#2 \global #1=#2 \fi}
 
\def\putbox(#1,#2)#3{%
\horsize{\wideness}{#3} \divide\wideness by 2
{\advance\wideness by #1 \ratchet{\rightwidth}{\wideness}}
{\advance\wideness by -#1 \ratchet{\leftwidth}{\wideness}}
\vertsize{\highness}{#3} \divide\highness by 2
{\advance\highness by #2 \ratchet{\topheight}{\highness}}
{\advance\highness by -#2 \ratchet{\botheight}{\highness}}
\put(#1,#2){\makebox(0,0){$#3$}}}
 
\def\putlbox(#1,#2)#3{%
\horsize{\wideness}{#3}
{\advance\wideness by #1 \ratchet{\rightwidth}{\wideness}}
{\ratchet{\leftwidth}{-#1}}
\vertsize{\highness}{#3} \divide\highness by 2
{\advance\highness by #2 \ratchet{\topheight}{\highness}}
{\advance\highness by -#2 \ratchet{\botheight}{\highness}}
\put(#1,#2){\makebox(0,0)[l]{$#3$}}}
 
\def\putrbox(#1,#2)#3{%
\horsize{\wideness}{#3}
{\ratchet{\rightwidth}{#1}}
{\advance\wideness by -#1 \ratchet{\leftwidth}{\wideness}}
\vertsize{\highness}{#3} \divide\highness by 2
{\advance\highness by #2 \ratchet{\topheight}{\highness}}
{\advance\highness by -#2 \ratchet{\botheight}{\highness}}
\put(#1,#2){\makebox(0,0)[r]{$#3$}}}

\def\adjust[#1]{} 
 
\newcount \coefa
\newcount \coefb
\newcount \coefc
\newcount\tempcounta
\newcount\tempcountb
\newcount\tempcountc
\newcount\tempcountd
\newcount\xext
\newcount\yext
\newcount\xoff
\newcount\yoff
\newcount\gap%
\newcount\arrowtypea
\newcount\arrowtypeb
\newcount\arrowtypec
\newcount\arrowtyped
\newcount\arrowtypee
\newcount\height
\newcount\width
\newcount\xpos
\newcount\ypos
\newcount\run
\newcount\rise
\newcount\arrowlength
\newcount\halflength
\newcount\arrowtype
\newdimen\tempdimen
\newdimen\xlen
\newdimen\ylen
\newsavebox{\tempboxa}%
\newsavebox{\tempboxb}%
\newsavebox{\tempboxc}%
 
\newdimen\w@dth
 
\def\setw@dth#1#2{\setbox\z@\hbox{\m@th$#1$}\w@dth=\wd\z@
\setbox\@ne\hbox{\m@th$#2$}\ifnum\w@dth<\wd\@ne \w@dth=\wd\@ne \fi
\advance\w@dth by 1.2em}
 
\def\t@^#1_#2{\allowbreak\def\n@one{#1}\def\n@two{#2}\mathrel
{\setw@dth{#1}{#2}
\mathop{\hbox to \w@dth{\rightarrowfill}}\limits
\ifx\n@one\empty\else ^{\box\z@}\fi
\ifx\n@two\empty\else _{\box\@ne}\fi}}
\def\t@@^#1{\@ifnextchar_{\t@^{#1}}{\t@^{#1}_{}}}
\def\to{\@ifnextchar^{\t@@}{\t@@^{}}}
 
\def\t@left^#1_#2{\def\n@one{#1}\def\n@two{#2}\mathrel{\setw@dth{#1}{#2}
\mathop{\hbox to \w@dth{\leftarrowfill}}\limits
\ifx\n@one\empty\else ^{\box\z@}\fi
\ifx\n@two\empty\else _{\box\@ne}\fi}}
\def\t@@left^#1{\@ifnextchar_{\t@left^{#1}}{\t@left^{#1}_{}}}
\def\toleft{\@ifnextchar^{\t@@left}{\t@@left^{}}}
 
\def\two@^#1_#2{\allowbreak
\def\n@one{#1}\def\n@two{#2}\mathrel{\setw@dth{#1}{#2}
\mathop{\vcenter{\lineskip\z@\baselineskip\z@
                 \hbox to \w@dth{\rightarrowfill}%
                 \hbox to \w@dth{\rightarrowfill}}%
       }\limits
\ifx\n@one\empty\else ^{\box\z@}\fi
\ifx\n@two\empty\else _{\box\@ne}\fi}}
\def\tw@@^#1{\@ifnextchar _{\two@^{#1}}{\two@^{#1}_{}}}
\def\two{\@ifnextchar ^{\tw@@}{\tw@@^{}}}
 
\def\tofr@^#1_#2{\def\n@one{#1}\def\n@two{#2}\mathrel{\setw@dth{#1}{#2}
\mathop{\vcenter{\hbox to \w@dth{\rightarrowfill}\kern-1.7ex
                 \hbox to \w@dth{\leftarrowfill}}%
       }\limits
\ifx\n@one\empty\else ^{\box\z@}\fi
\ifx\n@two\empty\else _{\box\@ne}\fi}}
\def\t@fr@^#1{\@ifnextchar_ {\tofr@^{#1}}{\tofr@^{#1}_{}}}
\def\tofro{\@ifnextchar^ {\t@fr@}{\t@fr@^{}}}

\def\mon{\mathop{\m@th\hbox to
      14.6\P@{\lasyb\char'51\hskip-2.1\P@$\arrext$\hss
$\mathord\rightarrow$}}\limits} 
\def\leftmono{\mathrel{\m@th\hbox to
14.6\P@{$\mathord\leftarrow$\hss$\arrext$\hskip-2.1\P@\lasyb\char'50%
}}\limits} 
\mathchardef\arrext="0200       

\def\settypes(#1,#2,#3){\arrowtypea#1 \arrowtypeb#2 \arrowtypec#3}
\def\settoheight#1#2{\setbox\@tempboxa\hbox{#2}#1\ht\@tempboxa\relax}%
\def\settodepth#1#2{\setbox\@tempboxa\hbox{#2}#1\dp\@tempboxa\relax}%
\def\settokens`#1`#2`#3`#4`{%
     \def\tokena{#1}\def\tokenb{#2}\def\tokenc{#3}\def\tokend{#4}}
\def\setsqparms[#1`#2`#3`#4;#5`#6]{%
\arrowtypea #1
\arrowtypeb #2
\arrowtypec #3
\arrowtyped #4
\width #5
\height #6
}
\def\setpos(#1,#2){\xpos=#1 \ypos#2}

\def\settriparms[#1`#2`#3;#4]{\settripairparms[#1`#2`#3`1`1;#4]}%
 
\def\settripairparms[#1`#2`#3`#4`#5;#6]{%
\arrowtypea #1
\arrowtypeb #2
\arrowtypec #3
\arrowtyped #4
\arrowtypee #5
\width #6
\height #6
}
 
\def\resetparms{\settripairparms[1`1`1`1`1;500]\width 500}
 
\resetparms
 
\def\mvector(#1,#2)#3{
\put(0,0){\vector(#1,#2){#3}}%
\put(0,0){\vector(#1,#2){26}}%
}
\def\evector(#1,#2)#3{{
\arrowlength #3
\put(0,0){\vector(#1,#2){\arrowlength}}%
\advance \arrowlength by-30
\put(0,0){\vector(#1,#2){\arrowlength}}%
}}
 
\def\horsize#1#2{%
\settowidth{\tempdimen}{$#2$}%
#1=\tempdimen
\divide #1 by\unitlength
}
 
\def\vertsize#1#2{%
\settoheight{\tempdimen}{$#2$}%
#1=\tempdimen
\settodepth{\tempdimen}{$#2$}%
\advance #1 by\tempdimen
\divide #1 by\unitlength
}
 
\def\putvector(#1,#2)(#3,#4)#5#6{{%
\ifnum3<\arrowtype
\putdashvector(#1,#2)(#3,#4)#5\arrowtype
\else
\ifnum\arrowtype<-3
\putdashvector(#1,#2)(#3,#4)#5\arrowtype
\else
\xpos=#1
\ypos=#2
\run=#3
\rise=#4
\arrowlength=#5
\ifnum \arrowtype<0
    \ifnum \run=0
        \advance \ypos by-\arrowlength
    \else
        \tempcounta \arrowlength
        \multiply \tempcounta by\rise
        \divide \tempcounta by\run
        \ifnum\run>0
            \advance \xpos by\arrowlength
            \advance \ypos by\tempcounta
        \else
            \advance \xpos by-\arrowlength
            \advance \ypos by-\tempcounta
        \fi
    \fi
    \multiply \arrowtype by-1
    \multiply \rise by-1
    \multiply \run by-1
\fi
\ifcase \arrowtype
\or \put(\xpos,\ypos){\vector(\run,\rise){\arrowlength}}%
\or \put(\xpos,\ypos){\mvector(\run,\rise)\arrowlength}%
\or \put(\xpos,\ypos){\evector(\run,\rise){\arrowlength}}%
\fi\fi\fi
}}
 
\def\putsplitvector(#1,#2)#3#4{
\xpos #1
\ypos #2
\arrowtype #4
\halflength #3
\arrowlength #3
\gap 140
\advance \halflength by-\gap
\divide \halflength by2
\ifnum\arrowtype>0
   \ifcase \arrowtype
   \or \put(\xpos,\ypos){\line(0,-1){\halflength}}%
       \advance\ypos by-\halflength
       \advance\ypos by-\gap
       \put(\xpos,\ypos){\vector(0,-1){\halflength}}%
   \or \put(\xpos,\ypos){\line(0,-1)\halflength}%
       \put(\xpos,\ypos){\vector(0,-1)3}%
       \advance\ypos by-\halflength
       \advance\ypos by-\gap
       \put(\xpos,\ypos){\vector(0,-1){\halflength}}%
   \or \put(\xpos,\ypos){\line(0,-1)\halflength}%
       \advance\ypos by-\halflength
       \advance\ypos by-\gap
       \put(\xpos,\ypos){\evector(0,-1){\halflength}}%
   \fi
\else \arrowtype=-\arrowtype
   \ifcase\arrowtype
   \or \advance \ypos by-\arrowlength
       \put(\xpos,\ypos){\line(0,1){\halflength}}%
       \advance\ypos by\halflength
       \advance\ypos by\gap
       \put(\xpos,\ypos){\vector(0,1){\halflength}}%
   \or \advance \ypos by-\arrowlength
       \put(\xpos,\ypos){\line(0,1)\halflength}%
       \put(\xpos,\ypos){\vector(0,1)3}%
       \advance\ypos by\halflength
       \advance\ypos by\gap
       \put(\xpos,\ypos){\vector(0,1){\halflength}}%
   \or \advance \ypos by-\arrowlength
       \put(\xpos,\ypos){\line(0,1)\halflength}%
       \advance\ypos by\halflength
       \advance\ypos by\gap
       \put(\xpos,\ypos){\evector(0,1){\halflength}}%
   \fi
\fi
}
 
\def\putmorphism(#1)(#2,#3)[#4`#5`#6]#7#8#9{{%
\run #2
\rise #3
\ifnum\rise=0
  \puthmorphism(#1)[#4`#5`#6]{#7}{#8}#9%
\else\ifnum\run=0
  \putvmorphism(#1)[#4`#5`#6]{#7}{#8}#9%
\else
\setpos(#1)%
\arrowlength #7
\arrowtype #8
\ifnum\run=0
\else\ifnum\rise=0
\else
\ifnum\run>0
    \coefa=1
\else
   \coefa=-1
\fi
\ifnum\arrowtype>0
   \coefb=0
   \coefc=-1
\else
   \coefb=\coefa
   \coefc=1
   \arrowtype=-\arrowtype
\fi
\width=2
\multiply \width by\run
\divide \width by\rise
\ifnum \width<0  \width=-\width\fi
\advance\width by60
\if l#9 \width=-\width\fi
\putbox(\xpos,\ypos){#4}
{\multiply \coefa by\arrowlength
\advance\xpos by\coefa
\multiply \coefa by\rise
\divide \coefa by\run
\advance \ypos by\coefa
\putbox(\xpos,\ypos){#5} }%
{\multiply \coefa by\arrowlength
\divide \coefa by2
\advance \xpos by\coefa
\advance \xpos by\width
\multiply \coefa by\rise
\divide \coefa by\run
\advance \ypos by\coefa
\if l#9%
   \putrbox(\xpos,\ypos){#6}%
\else\if r#9%
   \putlbox(\xpos,\ypos){#6}%
\fi\fi }%
{\multiply \rise by-\coefc
\multiply \run by-\coefc
\multiply \coefb by\arrowlength
\advance \xpos by\coefb
\multiply \coefb by\rise
\divide \coefb by\run
\advance \ypos by\coefb
\multiply \coefc by70
\advance \ypos by\coefc
\multiply \coefc by\run
\divide \coefc by\rise
\advance \xpos by\coefc
\multiply \coefa by140
\multiply \coefa by\run
\divide \coefa by\rise
\advance \arrowlength by\coefa
\ifcase\arrowtype
\or \put(\xpos,\ypos){\vector(\run,\rise){\arrowlength}}%
\or \put(\xpos,\ypos){\mvector(\run,\rise){\arrowlength}}%
\or \put(\xpos,\ypos){\evector(\run,\rise){\arrowlength}}%
\fi}\fi\fi\fi\fi}}

\newcount\numbdashes \newcount\lengthdash \newcount\increment
 
\def\howmanydashes{
\numbdashes=\arrowlength \lengthdash=40
\divide\numbdashes by \lengthdash
\lengthdash=\arrowlength
\divide\lengthdash by \numbdashes
\increment=\lengthdash
\multiply\lengthdash by 3
\divide\lengthdash by 5
}
 
\def\putdashvector(#1)(#2,#3)#4#5{%
\ifnum#3=0 \putdashhvector(#1){#4}#5
\else
\ifnum#2=0
\putdashvvector(#1){#4}#5\fi\fi}
 
\def\putdashhvector(#1,#2)#3#4{{%
\arrowlength=#3 \howmanydashes
\multiput(#1,#2)(\increment,0){\numbdashes}%
{\vrule height .4pt width \lengthdash\unitlength}
\arrowtype=#4 \xpos=#1
\ifnum\arrowtype<0 \advance\arrowtype by 7 \fi
\ifcase\arrowtype
\or \advance\xpos by 10
    \put(\xpos,#2){\vector(-1,0){\lengthdash}}
    \advance\xpos by 40
    \put(\xpos,#2){\vector(-1,0){\lengthdash}}
\or \advance \xpos by 10
    \put(\xpos,#2){\vector(-1,0){\lengthdash}}
    \advance\xpos by  \arrowlength
    \advance\xpos by  -50
    \put(\xpos,#2){\vector(-1,0){\lengthdash}}
\or \advance\xpos by 10
    \put(\xpos,#2){\vector(-1,0){\lengthdash}}
\or \advance\xpos by \arrowlength
    \advance\xpos by -\lengthdash
    \put(\xpos,#2){\vector(1,0){\lengthdash}}
\or {\advance\xpos by 10
    \put(\xpos,#2){\vector(1,0){\lengthdash}}}
    \advance\xpos by \arrowlength
    \advance\xpos by -\lengthdash
    \put(\xpos,#2){\vector(1,0){\lengthdash}}
\or \advance\xpos by \arrowlength
    \advance\xpos by -\lengthdash
    \put(\xpos,#2){\vector(1,0){\lengthdash}}
    \advance\xpos by -40
    \put(\xpos,#2){\vector(1,0){\lengthdash}}
   \fi
}}
 
\def\putdashvvector(#1,#2)#3#4{{%
\arrowlength=#3 \howmanydashes
\ypos=#2 \advance\ypos by -\arrowlength
\multiput(#1,#2)(0,\increment){\numbdashes}%
    {\vrule width .4pt height \lengthdash\unitlength}
\arrowtype=#4 \ypos=#2
\ifnum\arrowtype<0 \advance\arrowtype by 7 \fi
\ifcase\arrowtype
\or \advance\ypos by \arrowlength \advance\ypos by -40
    \put(#1,\ypos){\vector(0,1){\lengthdash}}
    \advance\ypos by -40
    \put(#1,\ypos){\vector(0,1){\lengthdash}}
\or \advance\ypos by 10
    \put(#1,\ypos){\vector(0,1){\lengthdash}}
    \advance\ypos by \arrowlength \advance\ypos by -40
    \put(#1,\ypos){\vector(0,1){\lengthdash}}
\or \advance\ypos by \arrowlength \advance\ypos by -40
    \put(#1,\ypos){\vector(0,1){\lengthdash}}
\or \advance\ypos by 10
    \put(#1,\ypos){\vector(0,-1){\lengthdash}}
\or \advance\ypos by 10
    \put(#1,\ypos){\vector(0,-1){\lengthdash}}
    \advance\ypos by \arrowlength \advance\ypos by -40
    \put(#1,\ypos){\vector(0,-1){\lengthdash}}
\or \advance\ypos by 10
    \put(#1,\ypos){\vector(0,-1){\lengthdash}}
    \advance\ypos by 40
    \put(#1,\ypos){\vector(0,-1){\lengthdash}}
\fi
}}
 
\def\puthmorphism(#1,#2)[#3`#4`#5]#6#7#8{{%
\xpos #1
\ypos #2
\width #6
\arrowlength #6
\arrowtype=#7
\putbox(\xpos,\ypos){#3\vphantom{#4}}%
{\advance \xpos by\arrowlength
\putbox(\xpos,\ypos){\vphantom{#3}#4}}%
\horsize{\tempcounta}{#3}%
\horsize{\tempcountb}{#4}%
\divide \tempcounta by2
\divide \tempcountb by2
\advance \tempcounta by30
\advance \tempcountb by30
\advance \xpos by\tempcounta
\advance \arrowlength by-\tempcounta
\advance \arrowlength by-\tempcountb
\putvector(\xpos,\ypos)(1,0)\arrowlength\arrowtype
\divide \arrowlength by2
\advance \xpos by\arrowlength
\vertsize{\tempcounta}{#5}%
\divide\tempcounta by2
\advance \tempcounta by20
\if a#8 %
   \advance \ypos by\tempcounta
   \putbox(\xpos,\ypos){#5}%
\else
   \advance \ypos by-\tempcounta
   \putbox(\xpos,\ypos){#5}%
\fi}}
 
\def\putvmorphism(#1,#2)[#3`#4`#5]#6#7#8{{%
\xpos #1
\ypos #2
\arrowlength #6
\arrowtype #7
\settowidth{\xlen}{$#5$}%
\putbox(\xpos,\ypos){#3}%
{\advance \ypos by-\arrowlength
\putbox(\xpos,\ypos){#4}}%
{\advance\arrowlength by-140
\advance \ypos by-70
\ifdim\xlen>0pt
   \if m#8%
      \putsplitvector(\xpos,\ypos)\arrowlength\arrowtype
   \else
   \putvector(\xpos,\ypos)(0,-1)\arrowlength\arrowtype
   \fi
\else
   \putvector(\xpos,\ypos)(0,-1)\arrowlength\arrowtype
\fi}%
\ifdim\xlen>0pt
   \divide \arrowlength by2
   \advance\ypos by-\arrowlength
   \if l#8%
      \advance \xpos by-40
      \putrbox(\xpos,\ypos){#5}%
   \else\if r#8%
      \advance \xpos by40
      \putlbox(\xpos,\ypos){#5}%
   \else
      \putbox(\xpos,\ypos){#5}%
   \fi\fi
\fi
}}
 
\def\putsquarep<#1>(#2)[#3;#4`#5`#6`#7]{{%
\setsqparms[#1]%
\setpos(#2)%
\settokens`#3`%
\puthmorphism(\xpos,\ypos)[\tokenc`\tokend`{#7}]{\width}{\arrowtyped}b%
\advance\ypos by \height
\puthmorphism(\xpos,\ypos)[\tokena`\tokenb`{#4}]{\width}{\arrowtypea}a%
\putvmorphism(\xpos,\ypos)[``{#5}]{\height}{\arrowtypeb}l%
\advance\xpos by \width
\putvmorphism(\xpos,\ypos)[``{#6}]{\height}{\arrowtypec}r%
}}
 
\def\putsquare{\@ifnextchar <{\putsquarep}{\putsquarep%
   <\arrowtypea`\arrowtypeb`\arrowtypec`\arrowtyped;\width`\height>}}
\def\square{\@ifnextchar< {\squarep}{\squarep
   <\arrowtypea`\arrowtypeb`\arrowtypec`\arrowtyped;\width`\height>}}
\def\squarep<#1>[#2`#3`#4`#5;#6`#7`#8`#9]{{
\setsqparms[#1]
\diagram
\putsquarep<\arrowtypea`\arrowtypeb`\arrowtypec`
\arrowtyped;\width`\height>
(0,0)[#2`#3`#4`{#5};#6`#7`#8`{#9}]
\enddiagram
}}                                                 
\def\putptrianglep<#1>(#2,#3)[#4`#5`#6;#7`#8`#9]{{%
\settriparms[#1]%
\xpos=#2 \ypos=#3
\advance\ypos by \height
\puthmorphism(\xpos,\ypos)[#4`#5`{#7}]{\height}{\arrowtypea}a%
\putvmorphism(\xpos,\ypos)[`#6`{#8}]{\height}{\arrowtypeb}l%
\advance\xpos by\height
\putmorphism(\xpos,\ypos)(-1,-1)[``{#9}]{\height}{\arrowtypec}r%
}}
 
\def\putptriangle{\@ifnextchar <{\putptrianglep}{\putptrianglep
   <\arrowtypea`\arrowtypeb`\arrowtypec;\height>}}
\def\ptriangle{\@ifnextchar <{\ptrianglep}{\ptrianglep
   <\arrowtypea`\arrowtypeb`\arrowtypec;\height>}}
\def\ptrianglep<#1>[#2`#3`#4;#5`#6`#7]{{
\settriparms[#1]
\diagram
\putptrianglep<\arrowtypea`\arrowtypeb`
\arrowtypec;\height>
(0,0)[#2`#3`#4;#5`#6`{#7}]
\enddiagram
}}                                            
 
\def\putqtrianglep<#1>(#2,#3)[#4`#5`#6;#7`#8`#9]{{%
\settriparms[#1]%
\xpos=#2 \ypos=#3
\advance\ypos by\height
\puthmorphism(\xpos,\ypos)[#4`#5`{#7}]{\height}{\arrowtypea}a%
\putmorphism(\xpos,\ypos)(1,-1)[``{#8}]{\height}{\arrowtypeb}l%
\advance\xpos by\height
\putvmorphism(\xpos,\ypos)[`#6`{#9}]{\height}{\arrowtypec}r%
}}
 
\def\putqtriangle{\@ifnextchar <{\putqtrianglep}{\putqtrianglep
   <\arrowtypea`\arrowtypeb`\arrowtypec;\height>}}
\def\qtriangle{\@ifnextchar <{\qtrianglep}{\qtrianglep
   <\arrowtypea`\arrowtypeb`\arrowtypec;\height>}}
\def\qtrianglep<#1>[#2`#3`#4;#5`#6`#7]{{
\settriparms[#1]
\width=\height                                
\diagram
\putqtrianglep<\arrowtypea`\arrowtypeb`
\arrowtypec;\height>
(0,0)[#2`#3`#4;#5`#6`{#7}]
\enddiagram
}}
 
\def\putdtrianglep<#1>(#2,#3)[#4`#5`#6;#7`#8`#9]{{%
\settriparms[#1]%
\xpos=#2 \ypos=#3
\puthmorphism(\xpos,\ypos)[#5`#6`{#9}]{\height}{\arrowtypec}b%
\advance\xpos by \height \advance\ypos by\height
\putmorphism(\xpos,\ypos)(-1,-1)[``{#7}]{\height}{\arrowtypea}l%
\putvmorphism(\xpos,\ypos)[#4``{#8}]{\height}{\arrowtypeb}r%
}}
 
\def\putdtriangle{\@ifnextchar <{\putdtrianglep}{\putdtrianglep
   <\arrowtypea`\arrowtypeb`\arrowtypec;\height>}}
\def\dtriangle{\@ifnextchar <{\dtrianglep}{\dtrianglep
   <\arrowtypea`\arrowtypeb`\arrowtypec;\height>}}
\def\dtrianglep<#1>[#2`#3`#4;#5`#6`#7]{{
\settriparms[#1]
\width=\height                                
\diagram
\putdtrianglep<\arrowtypea`\arrowtypeb`
\arrowtypec;\height>
(0,0)[#2`#3`#4;#5`#6`{#7}]
\enddiagram
}}
 
\def\putbtrianglep<#1>(#2,#3)[#4`#5`#6;#7`#8`#9]{{%
\settriparms[#1]%
\xpos=#2 \ypos=#3
\puthmorphism(\xpos,\ypos)[#5`#6`{#9}]{\height}{\arrowtypec}b%
\advance\ypos by\height
\putmorphism(\xpos,\ypos)(1,-1)[``{#8}]{\height}{\arrowtypeb}r%
\putvmorphism(\xpos,\ypos)[#4``{#7}]{\height}{\arrowtypea}l%
}}
 
\def\putbtriangle{\@ifnextchar <{\putbtrianglep}{\putbtrianglep
   <\arrowtypea`\arrowtypeb`\arrowtypec;\height>}}
\def\btriangle{\@ifnextchar <{\btrianglep}{\btrianglep
   <\arrowtypea`\arrowtypeb`\arrowtypec;\height>}}
\def\btrianglep<#1>[#2`#3`#4;#5`#6`#7]{{
\settriparms[#1]
\width=\height                               
\diagram
\putbtrianglep<\arrowtypea`\arrowtypeb`
\arrowtypec;\height>
(0,0)[#2`#3`#4;#5`#6`{#7}]
\enddiagram
}}
 
\def\putAtrianglep<#1>(#2,#3)[#4`#5`#6;#7`#8`#9]{{%
\settriparms[#1]%
\xpos=#2 \ypos=#3
{\multiply \height by2
\puthmorphism(\xpos,\ypos)[#5`#6`{#9}]{\height}{\arrowtypec}b}%
\advance\xpos by\height \advance\ypos by\height
\putmorphism(\xpos,\ypos)(-1,-1)[#4``{#7}]{\height}{\arrowtypea}l%
\putmorphism(\xpos,\ypos)(1,-1)[``{#8}]{\height}{\arrowtypeb}r%
}}
 
\def\putAtriangle{\@ifnextchar <{\putAtrianglep}{\putAtrianglep
   <\arrowtypea`\arrowtypeb`\arrowtypec;\height>}}
\def\Atriangle{\@ifnextchar <{\Atrianglep}{\Atrianglep
   <\arrowtypea`\arrowtypeb`\arrowtypec;\height>}}
\def\Atrianglep<#1>[#2`#3`#4;#5`#6`#7]{{
\settriparms[#1]
\width=\height                                     
\diagram
\putAtrianglep<\arrowtypea`\arrowtypeb`
\arrowtypec;\height>
(0,0)[#2`#3`#4;#5`#6`{#7}]
\enddiagram
}}
 
\def\putAtrianglepairp<#1>(#2)[#3;#4`#5`#6`#7`#8]{{%
\settripairparms[#1]%
\setpos(#2)%
\settokens`#3`%
\puthmorphism(\xpos,\ypos)[\tokenb`\tokenc`{#7}]{\height}{\arrowtyped}b%
\advance\xpos by\height
\puthmorphism(\xpos,\ypos)[\phantom{\tokenc}`\tokend`{#8}]%
{\height}{\arrowtypee}b%
\advance\ypos by\height
\putmorphism(\xpos,\ypos)(-1,-1)[\tokena``{#4}]{\height}{\arrowtypea}l%
\putvmorphism(\xpos,\ypos)[``{#5}]{\height}{\arrowtypeb}m%
\putmorphism(\xpos,\ypos)(1,-1)[``{#6}]{\height}{\arrowtypec}r%
}}
 
\def\putAtrianglepair{\@ifnextchar <{\putAtrianglepairp}{\putAtrianglepairp%
   <\arrowtypea`\arrowtypeb`\arrowtypec`\arrowtyped`\arrowtypee;\height>}}
\def\Atrianglepair{\@ifnextchar <{\Atrianglepairp}{\Atrianglepairp%
   <\arrowtypea`\arrowtypeb`\arrowtypec`\arrowtyped`\arrowtypee;\height>}}
 
\def\Atrianglepairp<#1>[#2;#3`#4`#5`#6`#7]{{
\settripairparms[#1]
\settokens`#2`
\width=\height                                
\diagram
\putAtrianglepairp                            
<\arrowtypea`\arrowtypeb`\arrowtypec`
\arrowtyped`\arrowtypee;\height>
(0,0)[{#2};#3`#4`#5`#6`{#7}]
\enddiagram
}}
 
\def\putVtrianglep<#1>(#2,#3)[#4`#5`#6;#7`#8`#9]{{%
\settriparms[#1]%
\xpos=#2 \ypos=#3
\advance\ypos by\height
{\multiply\height by2
\puthmorphism(\xpos,\ypos)[#4`#5`{#7}]{\height}{\arrowtypea}a}%
\putmorphism(\xpos,\ypos)(1,-1)[`#6`{#8}]{\height}{\arrowtypeb}l%
\advance\xpos by\height
\advance\xpos by\height
\putmorphism(\xpos,\ypos)(-1,-1)[``{#9}]{\height}{\arrowtypec}r%
}}
 
\def\putVtriangle{\@ifnextchar <{\putVtrianglep}{\putVtrianglep
   <\arrowtypea`\arrowtypeb`\arrowtypec;\height>}}
\def\Vtriangle{\@ifnextchar <{\Vtrianglep}{\Vtrianglep
   <\arrowtypea`\arrowtypeb`\arrowtypec;\height>}}
\def\Vtrianglep<#1>[#2`#3`#4;#5`#6`#7]{{
\settriparms[#1]
\width=\height                                 
\diagram
\putVtrianglep<\arrowtypea`\arrowtypeb`
\arrowtypec;\height>
(0,0)[#2`#3`#4;#5`#6`{#7}]
\enddiagram
}}
 
\def\putVtrianglepairp<#1>(#2)[#3;#4`#5`#6`#7`#8]{{
\settripairparms[#1]%
\setpos(#2)%
\settokens`#3`%
\advance\ypos by\height
\putmorphism(\xpos,\ypos)(1,-1)[`\tokend`{#6}]{\height}{\arrowtypec}l%
\puthmorphism(\xpos,\ypos)[\tokena`\tokenb`{#4}]{\height}{\arrowtypea}a%
\advance\xpos by\height
\puthmorphism(\xpos,\ypos)[\phantom{\tokenb}`\tokenc`{#5}]%
{\height}{\arrowtypeb}a%
\putvmorphism(\xpos,\ypos)[``{#7}]{\height}{\arrowtyped}m%
\advance\xpos by\height
\putmorphism(\xpos,\ypos)(-1,-1)[``{#8}]{\height}{\arrowtypee}r%
}}
 
\def\putVtrianglepair{\@ifnextchar <{\putVtrianglepairp}{\putVtrianglepairp%
    <\arrowtypea`\arrowtypeb`\arrowtypec`\arrowtyped`\arrowtypee;\height>}}
\def\Vtrianglepair{\@ifnextchar <{\Vtrianglepairp}{\Vtrianglepairp%
    <\arrowtypea`\arrowtypeb`\arrowtypec`\arrowtyped`\arrowtypee;\height>}}
\def\Vtrianglepairp<#1>[#2;#3`#4`#5`#6`#7]{{
\settripairparms[#1]
\settokens`#2`
\diagram
\putVtrianglepairp                             
<\arrowtypea`\arrowtypeb`\arrowtypec`
\arrowtyped`\arrowtypee;\height>
(0,0)[{#2};#3`#4`#5`#6`{#7}]
\enddiagram
}}

\def\putCtrianglep<#1>(#2,#3)[#4`#5`#6;#7`#8`#9]{{%
\settriparms[#1]%
\xpos=#2 \ypos=#3
\advance\ypos by\height
\putmorphism(\xpos,\ypos)(1,-1)[``{#9}]{\height}{\arrowtypec}l%
\advance\xpos by\height
\advance\ypos by\height
\putmorphism(\xpos,\ypos)(-1,-1)[#4`#5`{#7}]{\height}{\arrowtypea}l%
{\multiply\height by 2
\putvmorphism(\xpos,\ypos)[`#6`{#8}]{\height}{\arrowtypeb}r}%
}}
 
\def\putCtriangle{\@ifnextchar <{\putCtrianglep}{\putCtrianglep
    <\arrowtypea`\arrowtypeb`\arrowtypec;\height>}}
\def\Ctriangle{\@ifnextchar <{\Ctrianglep}{\Ctrianglep
    <\arrowtypea`\arrowtypeb`\arrowtypec;\height>}}
\def\Ctrianglep<#1>[#2`#3`#4;#5`#6`#7]{{
\settriparms[#1]
\width=\height                               
\diagram
\putCtrianglep<\arrowtypea`\arrowtypeb`
\arrowtypec;\height>
(0,0)[#2`#3`#4;#5`#6`{#7}]
\enddiagram
}}                                           
\def\putDtrianglep<#1>(#2,#3)[#4`#5`#6;#7`#8`#9]{{%
\settriparms[#1]%
\xpos=#2 \ypos=#3
\advance\xpos by\height \advance\ypos by\height
\putmorphism(\xpos,\ypos)(-1,-1)[``{#9}]{\height}{\arrowtypec}r%
\advance\xpos by-\height \advance\ypos by\height
\putmorphism(\xpos,\ypos)(1,-1)[`#5`{#8}]{\height}{\arrowtypeb}r%
{\multiply\height by 2
\putvmorphism(\xpos,\ypos)[#4`#6`{#7}]{\height}{\arrowtypea}l}%
}}
 
\def\putDtriangle{\@ifnextchar <{\putDtrianglep}{\putDtrianglep
    <\arrowtypea`\arrowtypeb`\arrowtypec;\height>}}
\def\Dtriangle{\@ifnextchar <{\Dtrianglep}{\Dtrianglep
   <\arrowtypea`\arrowtypeb`\arrowtypec;\height>}}
\def\Dtrianglep<#1>[#2`#3`#4;#5`#6`#7]{{
\settriparms[#1]
\width=\height                              
\diagram
\putDtrianglep<\arrowtypea`\arrowtypeb`
\arrowtypec;\height>
(0,0)[#2`#3`#4;#5`#6`{#7}]
\enddiagram
}}                                          
\def\setrecparms[#1`#2]{\width=#1 \height=#2}%
 
\def\recursep<#1`#2>[#3;#4`#5`#6`#7`#8]{{\m@th
\width=#1 \height=#2
\settokens`#3`
\settowidth{\tempdimen}{$\tokena$}
\ifdim\tempdimen=0pt
  \savebox{\tempboxa}{\hbox{$\tokenb$}}%
  \savebox{\tempboxb}{\hbox{$\tokend$}}%
  \savebox{\tempboxc}{\hbox{$#6$}}%
\else
  \savebox{\tempboxa}{\hbox{$\hbox{$\tokena$}\times\hbox{$\tokenb$}$}}%
  \savebox{\tempboxb}{\hbox{$\hbox{$\tokena$}\times\hbox{$\tokend$}$}}%
  \savebox{\tempboxc}{\hbox{$\hbox{$\tokena$}\times\hbox{$#6$}$}}%
\fi
\ypos=\height
\divide\ypos by 2
\xpos=\ypos
\advance\xpos by \width
\bfig
\putCtrianglep<-1`1`1;\ypos>(0,0)[`\tokenc`;#5`#6`{#7}]%
\puthmorphism(\ypos,0)[\tokend`\usebox{\tempboxb}`{#8}]{\width}{-1}b%
\puthmorphism(\ypos,\height)[\tokenb`\usebox{\tempboxa}`{#4}]{\width}{-1}a%
\advance\ypos by \width
\putvmorphism(\ypos,\height)[``\usebox{\tempboxc}]{\height}1r%
\efig
}}
 
\def\recurse{\@ifnextchar <{\recursep}{\recursep<\width`\height>}}
 
\def\puttwohmorphisms(#1,#2)[#3`#4;#5`#6]#7#8#9{{%
\puthmorphism(#1,#2)[#3`#4`]{#7}0a
\ypos=#2
\advance\ypos by 20
\puthmorphism(#1,\ypos)[\phantom{#3}`\phantom{#4}`#5]{#7}{#8}a
\advance\ypos by -40
\puthmorphism(#1,\ypos)[\phantom{#3}`\phantom{#4}`#6]{#7}{#9}b
}}
 
\def\puttwovmorphisms(#1,#2)[#3`#4;#5`#6]#7#8#9{{%
\putvmorphism(#1,#2)[#3`#4`]{#7}0a
\xpos=#1
\advance\xpos by -20
\putvmorphism(\xpos,#2)[\phantom{#3}`\phantom{#4}`#5]{#7}{#8}l
\advance\xpos by 40
\putvmorphism(\xpos,#2)[\phantom{#3}`\phantom{#4}`#6]{#7}{#9}r
}}
 
\def\puthcoequalizer(#1)[#2`#3`#4;#5`#6`#7]#8#9{{%
\setpos(#1)%
\puttwohmorphisms(\xpos,\ypos)[#2`#3;#5`#6]{#8}11%
\advance\xpos by #8
\puthmorphism(\xpos,\ypos)[\phantom{#3}`#4`#7]{#8}1{#9}
}}
 
\def\putvcoequalizer(#1)[#2`#3`#4;#5`#6`#7]#8#9{{%
\setpos(#1)%
\puttwovmorphisms(\xpos,\ypos)[#2`#3;#5`#6]{#8}11%
\advance\ypos by -#8
\putvmorphism(\xpos,\ypos)[\phantom{#3}`#4`#7]{#8}1{#9}
}}
 
\def\putthreehmorphisms(#1)[#2`#3;#4`#5`#6]#7(#8)#9{{%
\setpos(#1) \settypes(#8)
\if a#9 %
     \vertsize{\tempcounta}{#5}%
     \vertsize{\tempcountb}{#6}%
     \ifnum \tempcounta<\tempcountb \tempcounta=\tempcountb \fi
\else
     \vertsize{\tempcounta}{#4}%
     \vertsize{\tempcountb}{#5}%
     \ifnum \tempcounta<\tempcountb \tempcounta=\tempcountb \fi
\fi
\advance \tempcounta by 60
\puthmorphism(\xpos,\ypos)[#2`#3`#5]{#7}{\arrowtypeb}{#9}
\advance\ypos by \tempcounta
\puthmorphism(\xpos,\ypos)[\phantom{#2}`\phantom{#3}`#4]{#7}{\arrowtypea}{#9}
\advance\ypos by -\tempcounta \advance\ypos by -\tempcounta
\puthmorphism(\xpos,\ypos)[\phantom{#2}`\phantom{#3}`#6]{#7}{\arrowtypec}{#9}
}}
 
\def\setarrowtoks[#1`#2`#3`#4`#5`#6]{%
\def\toka{#1}
\def\tokb{#2}
\def\tokc{#3}
\def\tokd{#4}
\def\toke{#5}
\def\tokf{#6}
}
\def\hex{\@ifnextchar <{\hexp}{\hexp<1000`400>}}
\def\hexp<#1`#2>[#3`#4`#5`#6`#7`#8;#9]{%
\setarrowtoks[#9]
\yext=#2 \advance \yext by #2
\xext=#1 \advance\xext by \yext
\bfig
\putCtriangle<-1`0`1;#2>(0,0)[`#5`;\tokb``\tokd]
\xext=#1 \yext=#2 \advance \yext by #2
\putsquare<1`0`0`1;\xext`\yext>(#2,0)[#3`#4`#7`#8;\toka```\tokf]
\advance \xext by #2
\putDtriangle<0`1`-1;#2>(\xext,0)[`#6`;`\tokc`\toke]
\efig
}
\makeatother

\newsymbol\boxtimes 1202
\newsymbol\subsetneqq 2324
\newsymbol\lneqq 2308
\newsymbol\subseteqq 136A
\usepackage{amsmath}
\usepackage{amssymb}
\usepackage{amscd}
\usepackage{latexsym}

\newtheorem{thm}{Theorem}[section]
\newtheorem{cor}[thm]{Corollary}
\newtheorem{lma}[thm]{Lemma}
\newtheorem{prop}[thm]{Proposition}
\newtheorem{propdefn}[thm]{Proposition and Definition}
\theoremstyle{definition}
\newtheorem{defn}[thm]{Definition}
\newtheorem{defnrem}[thm]{Definition and Remark}
\newtheorem{eg}[thm]{Example}

\newtheorem{rem}[thm]{Remark}
\newtheorem{notatn}[thm]{Notation}

\theoremstyle{remark}

\DeclareMathOperator{\Bal}{Bal}
\DeclareMathOperator{\End}{End}
\DeclareMathOperator{\Hom}{Hom}
\newcommand{\Mod}{\mbox{\rm Mod}}

\DeclareMathOperator{\Mon}{Mon}
\DeclareMathOperator{\sRg}{sRg}

\DeclareMathOperator{\xsMod}{\mbox{-\,} sMod}
\DeclareMathOperator{\Rel}{Rel}

\DeclareMathOperator{\Set}{Set}
\DeclareMathOperator{\id}{{id}}

\newcommand{\tensor}{\otimes}
\newcommand{\iso}{\cong}

\newcommand{\x}{{\hbox{-}}}

\newcommand{\N}{{\mathbb N}}

\newcommand{\Pa}{{\mathcal P}}

\newcommand{\sst}{\scriptstyle }
\newcommand{\simz}{\!\sim}
\newcommand{\impl}{\Rightarrow}

\renewcommand{\sum}{{\Sigma}}

\title[Remarks on Semimodules]{Remarks on Semimodules}
\author{Bodo Pareigis}
\address{Mathematisches Institut\\Universit\"at M\"unchen\\ Theresienstra\ss e 39 \\80333 M\"unchen\\Germany}
\email{PAREIGIS@MATH.LMU.DE}
\author{Helmut R\"ohrl}
\address{9322 La Jolla Farms Rd.\\La Jolla, CA 92037\\USA}
\email{HROHRL@UCSD.EDU}

\begin{document}

\date{\today}
\begin{abstract} This is a study of universal problems for semimodules, in particular coequalizers, coproducts, and tensor products. Furthermore the structure theory of semiideals of the semiring of natural numbers is extended.
\end{abstract}

\maketitle

 \noindent{Key} Words: Semimodule, semiideal of $\mathbb Z^+$, congruence relation, coequalizer, tensor product.

  \section*{Introduction}  
 Our ongoing research and a recent question in a discussion group on the Internet \cite{JD} lead us to a more detailed study of semimodules which is presented here.

 Many of the recent publications on semimodules are devoted to transferring certain properties from modules to the more general case of semimodules. In this paper we want to show that semimodules can be unexpectedly different from modules with respect to certain properties.
 
  In this note we present congruence relations, coequalizers, coproducts, and colimits of semimodules, free semimodules, and tensor products of semimodules. We usually start with a universal problem and determine the inner algebraic structure of a universal solution, before we give explicit constructions and thus the proof of the existence. Most of these turn out to be quite different from what should be expected when considering the case of modules.

 One chapter is devoted to the study of semiideals in $\N_0$, the semiring of natural numbers (including 0). Theorem \ref{period} implies immediately a number of Lemmas and Theorems of \cite{AD}. We give more detailed structure theorems for semiideals of $\N_0$. It is known that each semiideal $M$ contains a subsemiideal of the form $dT_n$, where $T_n = \{n' \in \N_0 | n \leq n'\}$, and $d$ is the greatest common divisor of a generating set of $M$. We give an explicit formula for $n$, if $M$ is generated by two elements. 
  
 We also find a rather surprising statement about the uniqueness of certain generating systems, we show that each semiideal of $\N_0$ has a uniquely determined smallest set of generators. We present an algorithm to compute this smallest set of generators.
  
  This work is in progress and will be continued.

 \section{Preliminaries on Semimodules}
 \typeout{----------Section:  Preliminaries on Semimodules}
 \markright{Preliminaries on Semimodules} 


 \begin{defn}
 A {\em commutative monoid} $M$ is a set $M$ together with a composition 
$+: M \times M \ni (m,m') \mapsto m+m' \in M$ and an element $0 \in M$ such that
 \begin{enumerate}
 \item $(m + m') + m'' = m + (m' + m'')$,
 \item $m + m' = m' + m$,
 \item $0 + m = m$,
 \end{enumerate}
 for all $m,m', m'' \in M$.
 
 Let $M, N$ be commutative monoids. A {\em homomorphism} of commutative monoids is a map $f: M \to N$, such that
 \begin{enumerate}
 \item $f(m + m') = f(m) + f(m')$,
 \item $f(0) = 0$,
 \end{enumerate}
 for all $m, m' \in M$.
 
 The commutative monoids together with these homomorphisms form the category $\Mon$ of commutative monoids.

 We define\\
 $ \Hom(M,N) :=  $\\
 $\mbox{\phantom{XX}}\{ f: M \to N \vert f \mbox{ is 
a homomorphism of commutative monoids}\},$\\
 $\End(M) := \Hom(M,M).$
 \hfill $\Box$
 \end{defn}
 
 $\Hom(M,N)$ is a subset of $\Set(M,N) = N^M$, the set of maps from $M$ to $N$. $\Set(M,N) = N^M$ is a commutative monoid with componentwise addition inherited from $N$. $\Hom(M,N)$ is a submonoid of $\Set(M,N)$.
 
 \begin{defn}
 A {\em semiring} $R$ is a commutative monoid $(R,+,0)$ together with a composition 
written as multiplication $$\cdot: R \times R \ni (r,r') \mapsto rr' = r \cdot r' \in R$$ and an element 
$1 \in R$ such that
 \begin{enumerate}
 \item $(rr')r'' = r(r'r'')$,
 \item $ (r + r')r'' = rr'' + r'r'',$
 \item $r(r' + r'') = rr' + rr'',$
 \item $1\cdot r = r = r \cdot 1,$
 \item $0 \cdot r = 0 = r \cdot 0$
 \end{enumerate}
 for all $r,r', r'' \in R$.
 
  Let $R, S$ be semirings. A {\em homomorphism} of semirings is a map $f: R \to S$, such that
 \begin{enumerate}
 \item $f(r + r') = f(r) + f(r')$,
 \item $f(0) = 0$,
 \item $f(rr') = f(r)f(r')$,
 \item $f(1) = 1$,
 \end{enumerate}
 for all $r, r' \in R$.
 
 The semirings together with these homomorphisms form the category $\sRg$ of semirings.
 
 We define \\
 $\sRg(R,S) := \{ f: R \to S \vert f \mbox{ is 
a homomorphism of semirings}\}.$ 
 \hfill $\Box$
 \end{defn}
 
 The set $\N_0$ of natural numbers (including $0$) equipped with the usual addition and multiplication is a semiring. If $M$ is a commutative monoid then $\End(M)$ is a semiring with composition $\circ$ as multiplication.
 
 \begin{defn}
 Let $R$ be a semiring. 
A {\em left $R$-semimodule} ${}_RM$ is a commutative monoid $M$ 
together with a composition
 $$\cdot: R \times M \ni (r,m) \mapsto rm = r \cdot m\in M$$
 such that
 \begin{enumerate}
 \item $(rr')m = r(r'm)$,
 \item $(r + r')m = rm + r'm$,
 \item $r(m + m') = rm + rm'$,
 \item $1 \cdot m = m,$
 \item $0 \cdot m = 0 = r \cdot 0$
 \end{enumerate}
 for all $r,r' \in R$, $m,m' \in M$.

  A {\em homomorphism of left $R$-semimodules} $f: {}_RM \to {}_RN$ is a 
map $f:M \to N$ with 
 \begin{enumerate}
 \item $f(m + m') = f(m) + f(m')$,
 \item $f(rm) = rf(m)$
 \end{enumerate}
 for all $r \in R$, $m,m' \in M$.

 The left $R$-semimodules together with these homomorphisms form the category $R\xsMod$ of left  $R$-semimodules.

 {\em Right $R$-semimodules} and {\em homomorphisms of right $R$-semimodules} are 
defined analogously. 

 We define
 $$\begin{array}{l}
 \Hom_R(.M,.N) := \\
\mbox{\phantom{XX}}\{ f: {}_RM \to {}_RN \vert f \mbox{ is 
a homomorphism of left $R$-semimodules}\}\\
 \End_R(.M) := \Hom_R(.M,.M).
 \end{array}$$ 
 Similarly $\Hom_R(M.,N.)$ denotes the set of homomorphisms of right 
$R$-semimodules $M_R$ and $N_R$.

 We often omit the dot, marking the side on which the semiring acts on the semimodule, and simply
write $\Hom_R(M,N)$ instead of $\Hom_R(.M,.N)$ and similarly $\End_R(M)$ instead of $\End_R(.M)$.
 \hfill $\Box$
 \end{defn}

 \begin{rem}
 Observe that $f(0) = 0$ for each homomorphism of $R$-semimodules. 
 
 $\Hom_R(.M,.N)$ is a submonoid of $\Hom(M,N)$, the commutative mo\-noid of monoid homomorphisms from $M$ to $N$.
 \end{rem}

 \begin{rem} \label{char module}
 Let $R$ be a semiring and $M$ be a commutative monoid. Then 
there is a one-to-one correspondence between maps $f: R 
\times M \to M$ that make $M$ into a left $R$-semimodule and 
homomorphisms of semirings (always preserving the unit element) $g: 
R \to \End(M)$. 

 From here on we write '$R$-semimodule' instead of 'left $R$-semimodule'. 
 All results for left $R$-semimodules translate to right $R$-semimodules.
 \end{rem}

 \begin{lma}
 \begin{enumerate} 
 \item Let ${}_RM, {}_RN$ be $R$-semimodules. Then\break $\Hom_R (M,N)$ is 
a commutative monoid by $(f + g) (m) : = f (m) + g (m)$.
 \item Let $f:{}_RM' \to {}_RM$ and $g:{}_RN \to {}_RN'$ be homomorphisms 
of $R$-semimodules. Then 
 $$\Hom_R(f,g): \Hom_R(M,N) \ni h \mapsto ghf \in \Hom_R(M',N')$$
is a homomorphism of commutative monoids.
 \end{enumerate}
 \end{lma} 

 \begin{proof}
 (1) Since $N$ is a commutative monoid the set of maps $\Set (M,N)$ 
is also a commutative monoid. The set of monoid homomorphisms 
$\Hom (M,N)$ is a submonoid of $\Set (M,N)$ (observe that 
this holds only for {\em commutative} monoids). We show that $\Hom_R 
(M,N)$ is a submonoid of $\Hom (M,N)$. We must show 
that $f + g$ is a homomorphism of $R$-semimodules if $f$ and $g$ 
are. Obviously $f+g$ is a homomorphism of monoids. Furthermore 
we have $(f+g) (r m) = f (r m)+ g (r m ) = r f (m) + r g 
(m) = r (f(m)+ g(m)) = r (f+g) (m).$ Clearly the zero map $0: M \to N$ 
is also a homomorphism of $R$-semimodules.

 (2) obvious.
 \end{proof}

 \begin{rem} Let $f: M \to N$ be a homomorphism of $R$-semimodules. 
 $f$ is bijective (an isomorphism) if and only if there 
exists a homomorphism of $R$-semimodules $g: N \to M$ such that 
 $$fg = \id_N \mbox{ and } gf = \id_M.$$
 Furthermore $g$ is uniquely determined by $f$.
 \end{rem}

 \begin{rem}
 Each commutative monoid is an $\N_0$-semimodule in a unique way. Each 
homomorphism of commutative monoids is a homomorphism of $\N_0$-semimodules. 
 \end{rem}

 \begin{proof} 
 By Remark \ref{char module} we have to find a unique 
homomorphism of semirings $g: {\mathbb N_0} \to \End(M)$. This 
holds more generally. If $S$ is a semiring then there is a 
unique homomorphism of semirings $g: \N_0 \to S$. Since a  
homomorphism of semirings must preserve the unit we have $g(1) = 1$. 
Define $g(n) := 1 + \ldots + 1$ ($n$-times) for $n \geq 0$. 
Then it is easy to check that $g$ is a homomorphism of semirings
and it is obviously unique. This means that $M$ is an
 $\N_0$-semimodule by $nm = m + \ldots + m$ ($n$-times) for $n 
\geq 0$. 

 If $f: M \to N$ is a homomorphism of commutative monoids then 
$f(nm) = f(m + \ldots + m) = f(m) + \ldots + f(m) = nf(m)$. Hence $f$ 
is a homomorphism of $\N_0$-semimodules. 
 \end{proof}
 
 \begin{defn}
 Let $M$ be an $R$-semimodule. A subset $N \subseteq M$ is a {\em subsemimodule} of $M$, 
 if 
 \begin{enumerate}
 \item $\forall n, n' \in N: n+n' \in N$,
 \item $\forall n \in N, r \in R: rn \in N$.
 \end{enumerate}
 \end{defn}
 
 Clearly a subsemimodule $N$ of $M$ is itself an $R$-semimodule with the induced operations.

 \section{Congruence Relations on Semimodules}
 \subsection{\bf Congruence relations}\hskip-3mm. 
 \begin{defn}\label{congrel5}
 Let $\sim$ be an equivalence relation on $M$. Define
 $$\Rel_\sim := \{(m, m') \in M \times M \vert m \sim m' \}$$
 and $p_i: \Rel_\sim \to M$ by $p_1(m_1,m_2) := m_1$ and $p_2(m_1,m_2) := m_2$. 
 
 Observe that for all $x \in \Rel_\sim$ the equation $x = (p_1(x), p_2(x))$ holds.
 \end{defn}
 
 \begin{defn} \label{congrel}
 Let $M$ be an $R$-semimodule. A {\em congruence relation} (of left $R$-semimodules) $\sim$ on $M$ is an 
equivalence relation on $M$ satisfying
  $$m \sim m' \Rightarrow (m + n \sim m' + n \mbox { and } rm \sim rm')$$
for all $m,m',n \in M$ and all $r\in R$.
 \hfill $\Box$
 \end{defn}
 
 \begin{prop} \label{congrel6}
 Let $\sim$ be an equivalence relation and $\Rel_\sim$ together with $p_i: \Rel_\sim \to M, i=1,2$ be the associated relation (as given in Definition \ref{congrel5}). Then the following are equivalent:
 \begin{enumerate}
 \item $\sim$ is a congruence relation on $M$,
 \item $\Rel_\sim$ is an $R$-subsemimodule of $M \times M$ and $p_1, p_2$ are homomorphisms of $R$-semi\-modules.
 \end{enumerate}
 \end{prop}

 \begin{proof} Straightforward. \end{proof}
 
 \subsection{\bf Congruence relations and and their partitions} \label{partition} \hskip-3mm. \newline \vskip-3mm 
 It is well known that there is a one-to-one correspondence between equivalence relations $\sim$ on a set $M$ and partitions $\Pa$ of $M$ by 
 $$m \sim m'\iff \overline m = \overline {m'} \mbox{ in }\Pa.$$
 Each partition $\Pa$ on a set $M$ comes with a uniquely defined (canonical) surjective map $\nu: M \to \Pa$.
 
 \begin{prop} \label{congrel2}
 Let $M$ be an $R$-semimodule. Let $\sim$ be an equivalence relation on $M$ and let $M/\simz$ be the associated partition on $M$.

 \begin{enumerate}
 \item If $\sim$ is a congruence relation then $M/\simz$ is an $R$-semi\-module by
  $$\overline{m} + \overline{m'} := \overline{m + m'}$$
   and 
  $$r \overline{m} := \overline {rm}.$$
  Furthermore the canonical map $\nu: M \ni m \mapsto \overline{m} \in M/\simz$ 
 is a homomorphism of $R$-semimodules.

 \item If $M/\simz$ is an $R$-semimodule and the canonical map $\nu: M \ni m \mapsto \overline{m} \in M/\simz$ 
 is a homomorphism of $R$-semimodules, then $\sim$ is a congruence relation.

 \item The one-to-one correspondence between equivalence relations $\sim$ on $M$ and partitions $\Pa = M/\simz$ of $M$ restricts to a one-to-one correspondence between the congruence relations on $M$ and the $R$-semimodule structures on $M/\sim$ such that $\nu: M\to M/\simz$ is a homomorphism.
 \end{enumerate}

 \end{prop}
 
 \begin{proof} 
 (1) The main thing to show is that addition and multiplication on $M/\simz$ are well-defined. Given $m, m', m'' \in M$. Assume that $\overline m = \overline{m'}$ or equivalently $m \sim m'$. Then $m + m'' \sim m' + m''$ hence $\overline{m + m''} = \overline {m' + m''}$. Furthermore $rm \sim rm'$ hence $\overline{rm} = \overline{rm'}$.
 
 (2) Let $m \sim m'$, then $\nu(m) = \nu(m')$. Thus $\overline{m + m''} = \nu(m + m'') = \nu(m) + \nu(m'') = \nu(m') + \nu(m'') = \nu(m' + m'') = \overline{m' + m''}$, hence $m + m'' \sim m' + m''$. Similarly we get $rm \sim rm'$.
 
 (3) Straightforward.
 \end{proof}

 There are three important ways to construct congruence relations on $R$-semimodules which we will discuss in the following subsections.

 \subsection{\bf Congruence relations and homomorphisms}\hskip-3mm. 
 \begin{lma} \label{congrel2a}
 If $f: M \to N$ is a homomorphism of $R$-semimodules, then $\sim_f$ defined by 
 $$m \sim_f m' :\iff f(m) = f(m')$$ for all $m, m' \in M$ is a congruence relation on $M$.
 \end{lma}
 
 \begin{proof} Straightforward. \end{proof}
  
 \begin{thm}{\bf (Homomorphism Theorem)}\label{congrel2b}
  If $f: M \to N$ is a homomorphism of $R$-semimodules and $\sim$ is a congruence relation  
on $M$ such that
 $$m \sim m' \impl f(m) = f(m'), $$ 
 for all $m,m' \in M$, then there is a unique homomorphism $f':M/\simz \ \to N$ such that 
the following diagram commutes: 
 
 $$ \qtriangle[ M ` M/\simz ` \ N.; \nu ` f ` f' ] $$
 \end{thm}

 \begin{proof} Straightforward. \end{proof}
 
 \begin{prop}
  If $\sim$ is a congruence relation on $M$ and $f: M \to N$ is a homomorphism of $R$-semimodules such that
 $$f(m) = f(m') \impl m \sim m', $$ 
 for all $m,m' \in M$, then $$\forall m, m' \in M\ :\ m \sim_f m' \impl m \sim m'.$$ 
 
 Furthermore there is a unique homomorphism $\nu':M/\simz_f \ \to M/\simz$ such that the following diagram commutes: 
 
 $$ \qtriangle[ M ` M/\simz_f ` \ M/\simz.; \nu_f ` \nu ` \nu' ] $$
 \end{prop}

 \begin{proof} Straightforward. \end{proof}

 \subsection{\bf Congruence relations and pairs of homomorphisms}\hskip-3mm. \newline \vskip-3mm 
 In this section we want to construct a coequalizer for a pair of homomorphisms $f, g: N \to M$ of $R$-semimodules. The existence of such a coequalizer is guaranteed by the fact, that the category of semimodules is a category of equationally defined algebras. A discussion for this can be found in an Internet note \cite {JA2} published in 2008. We do not only want to have the proof of existence, but also a construction of coequalizers. 

 As a leading example we want to compute the coequalizer of the pair of homomorphisms 
$4\cdot, 6\cdot: \N_0 \to \N_0$. 

 \begin{lma} \label{congrrel4}
 Let $f, g: N \to M$ be two homomorphisms of $R$-semi\-modules. Let $\sim_{\langle f,g \rangle}$ be the relation on $M$ given by $m \sim_{\langle f,g \rangle} m'$ if and only if $h(m) = h(m')$ for all $R$-semimodules $P$ and all homomorphisms $h: M \to P,$ such that $h \circ f = h \circ g$ holds. Then $\sim_{\langle f,g \rangle}$ is a congruence relation.
 \end{lma}
 
 \begin{proof}
 Straightforward.
 \end{proof}
 
 \begin{thm}
 Let $f,g: N \to M$ be homomorphisms of $R$-semimo\-dules and let $\sim_{\langle f,g \rangle}$ be the congruence relation introduced in \ref{congrrel4}. Then $\nu: M \to M/\simz_{\langle f,g \rangle} $ is a coequalizer of $f,g$, i.e.
 \begin{enumerate}
 \item $\nu \circ f = \nu  \circ g$,
 \item for all $h: M \to P$ with $h \circ f = h \circ g$ there is a unique $h': M/\simz_{\langle f,g \rangle} \to P$ such that $h' \circ \nu = h$.
 \end{enumerate}
 \end{thm}

\begin{proof} 
 $C:= M/\simz_{\langle f,g \rangle}$ is an $R$-semimodule and the quotient map $\nu : M \to C$ is a homomorphism of left $R$-semimodules.  We show that $C$ together with $\nu : M \to C$ is a coequalizer of $f$ and $g$. 
 
 (1) Let $h: M \to P$ be a homomorphism with $h \circ f = h \circ g$. For $n \in N$ we get $h(f(n)) = h(g(n))$ hence $f(n) \sim_{\langle f,g \rangle} g(n)$ and thus $\nu \circ f = \nu  \circ g$.
 
 (2) Let $h:M \to P$ be a homomorphism satisfying $h \circ f = h \circ g$. Let $c = \nu (m) = \nu (m')$ or equivalently $m \sim_{\langle f,g \rangle} m'$. Then $h(m) = h(m')$. So $h$ furnishes a unique homomorphism $h' : C \to P$ with $h = h' \circ \nu$ which is the required factorization. Uniqueness of the factorization is clear, whence $C$ is a coequalizer.
 \end{proof}

 This construction of a coequalizer does not lend itself to an explicit construction in our example. So we will consider an interim approach modeled after the coequalizer construction in modules. If $f, g: N \to M$ are homomorphisms of modules, then the image of $f - g$ represents the set of elements congruent to zero in $M$. Equivalently two elements $m, m' \in M$ are congruent iff their difference is in the image of $f - g$ or $m \sim m' \iff m' - m = (f - g)(n)$ for some $n \in N$ or $m + f(n) = m' + g(n)$. In the case of semimodules  differences do not exist in general, so we replace $n$ by $x - y$. In a first attempt to get a suitable congruence relation for semimodules we use $m \sim m' \iff \exists x, y \in N: m + f(x) + g(y) = m' + f(y) + g(x)$.

 \begin{lma} \label{congrrel6}
 Let $f, g: N \to M$ be two homomorphisms of $R$-semimodules. Then\\ 
 $m \sim_{[f,g]} m' :\iff \exists n, n' \in N: m+f(n)+g(n') = m'+f(n')+g(n) $\\
 is a congruence relation. Furthermore $f(n) \sim_{[f,g]}  g(n)$ for all $n \in N$.
 \end{lma}

 \begin{proof}
 $m \sim_{[f,g]}  m$ and $m \sim_{[f,g]}  m' \impl m' \sim_{[f,g]}  m$ are clear. Assume $m \sim_{[f,g]}  m'$ and $m' \sim_{[f,g]}  m''$. Then there are $n, n',l, l' \in N$ with $m+f(n)+g(n') = m'+f(n')+g(n)$ and $m' + f(l)+g(l') = m'' + f(l')+g(l)$  hence
  $$\begin{array}{c}
  m + f(n + l)+g(n'+l') = m + f(n) + g(n') + f(l) + g(l') =\\ m' + f(n') + g(n) + f(l) + g(l') =\\ m'' + f(l') + g(l) + f(n') + g(n) = m'' + f(n'+l') + g(n + l)
  \end{array}$$
  so that $m \sim_{[f,g]}  m''$.
 
 Let $m \sim_{[f,g]}  m'$ with $m+f(n)+g(n') = m'+f(n')+g(n)$ and let $p \in M$. Then 
 $m + p + f(n) + g(n') = m' + p + f(n') + g(n)$, hence $m + p \sim_{[f,g]}  m' + p$. Furthermore
 $rm + f(rn) + g(rn') = r(m + f(n) + g(n')) = r(m' + f(n') + g(n)) = rm' + f(rn') + g(rn)$, hence $rm \sim_{[f,g]}  rm'$.

$f(n) + f(0) + g(n) = g(n) + f(n) + g(0)$ implies $f(n) \sim_{[f,g]}  g(n)$ for all $n \in N$.
\end{proof}

 We will see, however, that $h \circ f = h \circ g$ does not imply in general that $h(m) = h(m')$ for all $m, m'$ with $m \sim_{[f,g]} m'$. Indeed if $m \sim_{[f,g]}  m'$ then $m +f(n)+g(n') = m'+f(n')+g(n)$ for some $n, n' \in N$. If we apply $h$ to this equation we get $h(m) + hf(n) + hg(n') = h(m') + hf(n') + hg(n) = h(m') + hg(n') + hf(n)$. If $M$ if cancellable, then this implies $h(m) = h(m')$. But we will see that this does not hold in general. So our construction will not give a coequalizer.

 \begin{eg} \label{egcongrel1}
 Consider the two homomorphisms $4\cdot, 6\cdot: \N_0 \to \N_0$, given by the multiplication with 4 resp.\ 6. Then $0 \sim_{[f,g]} 2$ since $0 + 4 \cdot 0 + 6 \cdot 1 = 2 + 4 \cdot 1 + 6 \cdot 0$. So all even natural numbers are congruent to $0$. Since $1$ is not congruent to $0$, we get two distinct congruence classes $\bar 0$ and $\bar 1$. We will see further down that the coequalizer of the two homomorphisms is different from $\{ \bar 0, \bar 1 \}$ so even in the situation of a cancellable semimodule the above congruence relation does not give a coequalizer in the category of all semimodules. 
 \end{eg} 

  We consider now a third congruence relation.

 \begin{lma} \label{congrrel3}
 Let $f, g: N \to M$ be two homomorphisms of $R$-semimodules. Then 
 $$\begin{array}{l} 
     m \sim_{(f,g)} m' :\iff \\
     \exists m_1, \ldots, m_k \in M, n_1, \ldots, n_k, n'_1, \ldots, n'_k \in N: \\
     m = m_1 + f(n_1) + g(n'_1),\\
     m_1 + f(n'_1) + g(n_1) = m_2 + f(n_2) + g(n'_2), \\
     \ldots,\\
     m_{k-1} + f(n'_{k-1}) + g(n_{k-1}) = m_k + f(n_k) + g(n'_k), \\
     m_k + f(n'_k) + g(n_k) = m'
 \end{array}$$
 is a congruence relation.
 \end{lma}
 
 \begin{proof}
 Reflexivity of $\sim_{(f,g)}$ is clear with the empty chain. Symmetry follows from the symmetry of the equivalence chains. Transitivity is obtained by concatenating two such chains. The properties of a congruence relation are obtained by adding a summand from $M$ to all terms along an equivalence chain resp. by multiplying all terms with a factor from $R$.
 \end{proof}
 
 \begin{thm}\label{congrrel7}
 Let $f,g: N \to M$ be two homomorphisms of $R$-semimo\-dules and let $\sim_{(f,g)}$ be the congruence relation introduced in \ref{congrrel3}. Then $\nu: M \to M/\simz_{(f,g)}$ is a coequalizer of $f,g$, i.e.
 \begin{enumerate}
 \item $\nu \circ f = \nu \circ g$,
 \item for all $h: M \to P$ with $h \circ f = h \circ g$ there is a unique $h': M/\simz_{(f,g)} \to P$ such that $h' \circ \nu = h$.
 \end{enumerate}
 \end{thm}
 
 \begin{proof} (1) For $n \in N$ we have $f(n) = 0 + f(n) + g(0)$, $0 + f(0) + g(n) = g(n)$ hence $f(n) \sim_{(f,g)} g(n)$ and $\nu f(n) = \nu g(n)$.

 (2) Given $h$ with $hf = hg$. Define $h': M/\simz_{(f,g)} \to P$ by $h'(\overline m) := h(m)$. We have to show that $m \sim_{(f,g)} m'$ implies $h(m) = h(m')$, so that $h'$ is well-defined. Let $m  \sim_{(f,g)} m'$. Then $m = m_1 + f(n_1) + g(n'_1),
     m_1 + f(n'_1) + g(n_1) = m_2 + f(n_2) + g(n'_2), 
     \ldots,
     m_{k-1} + f(n'_{k-1}) + g(n_{k-1}) = m_k + f(n_k) + g(n'_k), 
     m_k + f(n'_k) + g(n_k) = m'$. 
We apply $h$ and get $h(m) = h(m_1) + hf(n_1) + hg(n'_1) = h(m_1) + hf(n'_1) + hg(n_1) = \ldots = h(m_k) + hf(n'_k) + hg(n_k) = h(m')$. By definition we have $h'\nu = h$. $\nu$ is surjective thus $h'$ is unique. Since addition and multiplication of the equivalence classes $\overline m$ are defined on the representatives $m$, it is clear that $h'$ is a homomorphism. 
 \end{proof}
 
 \begin{thm} \label{congrrel8}
 Let $f,g: N \to M$ be two homomorphisms of $R$-semimo\-dules. Then the two congruence relations $\sim_{\langle f,g \rangle}$ as defined in \ref{congrrel4}  and  $\sim_{(f,g)}$ as defined  in \ref{congrrel3} are equal.
 \end{thm}

 \begin{proof}
 Let $m \sim_{(f,g)} m'$. Let $h: M \to P$ be a homomorphism with $hf = hg$, then one shows as in the proof of Theorem \ref{congrrel7}, that $h(m) = h(m')$ hence $m \sim_{\langle f,g \rangle} m'$. Since the factorization homomorphism between the two coequalizers  $M/\sim_{(f,g)}$ and $M/\sim_{\langle f,g \rangle}$ is an isomorphism, we get that the two congruence relations are the same.
 \end{proof}

 \begin{eg} \label{congrrel9}
 As above we consider the two homomorphisms $4\cdot, 6\cdot: \N_0 \to \N_0$. Then by closer examination of the congruence relation $\sim_{(f,g)}$ we get
 $$\N_0/\simz_{(f,g)} = \{ \bar 0, \bar 1, \bar 2, \bar 3, \bar 4, \bar 5 \}.$$
 This together with Example \ref{egcongrel1} shows that the congruence relation in Lemma \ref{congrrel6} does not furnish the coequalizer. The addition in this semimodule is given by 
 $$\begin{tabular}{||c||c|c|c|c|c|c||} \hline \hline
 +        & $\bar 0$ & $\bar 1$ & $\bar 2$ & $\bar 3$ & $\bar 4$ & $\bar 5$\\ \hline \hline
 $\bar 0$ & $\bar 0$ & $\bar 1$ & $\bar 2$ & $\bar 3$ & $\bar 4$ & $\bar 5$\\ \hline 
 $\bar 1$ & $\bar 1$ & $\bar 2$ & $\bar 3$ & $\bar 4$ & $\bar 5$ & $\bar 4$\\ \hline 
 $\bar 2$ & $\bar 2$ & $\bar 3$ & $\bar 4$ & $\bar 5$ & $\bar 4$ & $\bar 5$\\ \hline 
 $\bar 3$ & $\bar 3$ & $\bar 4$ & $\bar 5$ & $\bar 4$ & $\bar 5$ & $\bar 4$\\ \hline 
 $\bar 4$ & $\bar 4$ & $\bar 5$ & $\bar 4$ & $\bar 5$ & $\bar 4$ & $\bar 5$\\ \hline 
 $\bar 5$ & $\bar 5$ & $\bar 4$ & $\bar 5$ & $\bar 4$ & $\bar 5$ & $\bar 4$\\ \hline \hline
 \end{tabular}$$
 \end{eg}

 \begin{lma}
 Let $\sim$ be a congruence relation on $M$, let $p_1, p_2: \Rel_\sim \to M$ be as in Definition \ref{congrel5}, and let $\sim_{(p_1,p_2)}$ be as in \ref{congrrel3}. Then
 $$ m \sim m' \impl m \sim_{(p_1,p_2)} m'$$
 for all $m, m' \in M$.
 \end{lma}

 \begin{proof} We have $m \sim m' \impl (m,m') \in \Rel_\sim \impl$ 
 $m = 0 + p_1(m,m') + p_2(0,0), 0 + p_1(0,0) + p_2(m,m') = m' \impl m \sim_{(p_1,p_2)} m'$. 
 \end{proof}

 \begin{lma}\label{pairhomo_cong} 
 Let $f,g: N \to M$ be homomorphisms of $R$-semimo\-dules and let $\sim$ be a congruence relation on $M$. Assume 
 $$f(n) \sim g(n) \mbox{ for all }n \in N.$$
 Then there is a unique homomorphism $h: N \to Rel_\sim$ such that the diagram
 $$\bfig 
 \putmorphism(480, 500)(0, -1)[``f]{500}1l 
 \putmorphism(500, 500)(0, -1)[N``]{500}0r
 \putmorphism(520, 500)(0, -1)[``g]{500}1r
 \putmorphism(0, 20)(1, 0)[\phantom{\Rel_\sim}`\phantom{M}`p_1]{500}1a
 \putmorphism(0, 0)(1, 0)[\Rel_\sim`M`]{500}0a
 \putmorphism(0, -20)(1, 0)[\phantom{\Rel_\sim}`\phantom{M}`p_2]{500}1b
 \putmorphism(500, 500)(-1, -1)[``h]{500}1l
 \efig$$
 commutes. 
 \end{lma}

 \begin{proof}
 Define $h(n) := (f(n), g(n))$. Then $h(n) \in \Rel_\sim$. Furthermore $h(n + n') = (f(n + n'), g(n + n')) = (f(n) + f(n'), g(n) + g(n')) = (f(n), g(n)) + (f(n'), g(n')) = h(n) + h(n')$ and similarly $h(rn) = rh(n)$. We also have $p_1h(n) = f(n)$ and $p_2h(n) = g(n)$. If $p_1 \psi = f$ and $p_2 \psi = g$, then $\psi(n) = (p_1\psi(n), p_2\psi(n)) = (f(n), g(n)) = h(n)$ hence $\psi = h$.
 \end{proof}

 \begin{thm} 
 Let $f: M \to N$ be a homomorphisms of $R$-semimo\-dules and let $\sim\ := \sim_f$ be the congruence relation introduced in Lemma \ref{congrel2a}. Then $p_1, p_2: \Rel_\sim \to M$ is a kernel pair of $f$, i.e.
 \begin{enumerate}
 \item $f p_1 = f p_2$,
 \item for all $g, h: P \to M$ with $f g = f h$ there is a unique $\mu: P \to \Rel_\sim$ such that $p_1\mu = g$ and $p_2\mu = h$,
  $$\bfig 
 \putmorphism(480, 500)(0, -1)[\phantom{P}`\phantom{M}`g]{500}1l 
 \putmorphism(500, 500)(0, -1)[P`\phantom{M}`]{500}0r
 \putmorphism(520, 500)(0, -1)[\phantom{P}`\phantom{M}`h]{500}1r
 \putmorphism(0, 20)(1, 0)[\phantom{\Rel_\sim}`\phantom{M}`p_1]{500}1a
 \putmorphism(0, 0)(1, 0)[\Rel_\sim`M`]{500}0a
 \putmorphism(0, -20)(1, 0)[\phantom{\Rel_\sim}`\phantom{M}`p_2]{500}1b
 \putmorphism(500, 500)(-1, -1)[\phantom{P}`\phantom{\Rel_\sim}`\mu]{500}1l
 \putmorphism(500, 0)(1, 0)[\phantom{M}`N`f]{500}1a
 \efig.$$
 \end{enumerate}
 \end{thm}
 
 \begin{proof} Straightforward. \end{proof}

Observe two known facts from category theory.
 \begin{enumerate} 
 \item The kernel pair of a homomorphism $f$ is the kernel pair of its coequalizer. 
 \item The coequalizer of a pair of homomorphisms $f,g$ is the coequalizer of its kernel pair.
 \end{enumerate}
 
 \subsection{\bf Congruence relations and subsemimodules}\hskip-3mm. 
 Let $K \subseteq M$ be an $R$-subsemimodule. Let furthermore $\iota: K \to M$ and $0: K \to M$ be the inclusion respectively the zero homomorphisms and apply to them the congruence relation defined in Lemma \ref{congrrel3}. It is easy to check that the congruence relation $m \simz_{(\iota,0)} m'$ can be simplified to \\
 $\phantom{XXXX}m \sim_K m' :\iff \exists a, b \in K : m+a = m'+b$\\
 called the Bourne relation in the literature \cite{JA1}. 
 
 \begin{prop} \label{congrel4}
 \begin{enumerate} 
 \item If $\ \sim\ $ is a congruence relation on $M$, then $K := \overline 0 = \{m \in M|m \sim 0\}$ is an 
$R$-subsemimodule of $M$.
 
 \item Let $K \subseteq M$ be an $R$-subsemimodule. Then
   $$m \sim_K m' :\iff \exists a, b \in K : m+a = m'+b$$
is a congruence relation. We write $M/K := M/{\sim_K}$. 
The quotient map $\nu:M \to M/{\sim_K}$ satisfies $\nu (K) = \{0\}$ and is a cokernel of $\iota$. Moreover, $\sim_K$ is the smallest congruence relation on $M$ whose associated quotient map sends $K$ to $\{0\}$. 
 \end{enumerate}
 \end{prop}
 
 \begin{proof} (1) Straightforward.\\
 (2) It is easy to verify that $\sim_K$ is a congruence relation whose quotient map $\nu $ satisfies $\nu (K) = \{0\}$. Let $\sim$ be a congruence relation on $M$ whose quotient map sends $K$ to $\{0\}$. Then $a \sim 0$ for all $a \in K$, whence $m \sim m + a$ for all $m \in M, a \in K$. Assume $m \sim_K m'$ for $m,m' \in M$. Then there are $a, b \in K$ such that $m + a = m' + b$. This implies $m \sim m+a = m'+b \sim m'$ and thus $\sim$ contains $\sim_K$.
 \end{proof}
 
 \begin{prop} \label{congrel4a}
 \begin{enumerate}
 \item Let $\sim$ be any congruence relation on $M$, let $K:= \overline 0$ and let $\sim_K$ be the congruence relation defined by $K$ as in \ref{congrel4} (2). Then for all $m, m' \in M$ 
   $$m \sim_K m' \impl m \sim m'.$$
   
 \item Let $K \subseteq M$ be an $R$-subsemimodule, $\sim_K$ be defined by $K$ and $L:= \overline 0$ the equivalence class of $0$ w.r.t. $\sim_K$. Then
    $$K \subseteq L.$$
 \end{enumerate}
 \end{prop}
 
 \begin{proof} Straightforward. \end{proof}

 \begin{prop}
 Let $f: M \to N$ be a homomorphism of $R$-semimodules and $\sim_f$ be the congruence relation as defined in \ref{congrel2a}. Then $f^{-1}(0)$ is an $R$-subsemimodule 
of $M$. For each $m \in M$ we have $m + f^{-1}(0) \subseteq \overline m$. Equality does not hold
in general.
 \end{prop}
 
 \begin{proof} Straightforward. \end{proof}

 \begin{prop}
 Let $K = f^{-1}(0)$ and $\sim_K$ be the congruence relation defined by $K$ as in \ref{congrel4} (2). Then $m \sim_K m' 
\impl m \sim_f m'$ for all $m, m' \in M$. So we get a surjective homomorphism of $R$-semimodules 
$\mu: M/K \to M/\simz_f$.
 \end{prop}

 \begin{proof} Straightforward. \end{proof}

 \section{Products and Coproducts of Semimodules}

 \begin{defn}
 \begin{enumerate}
 \item Let $(M_i | i \in I)$ be a family of $R$-semimodules. An 
$R$-semimodule $\prod M_i$ together with a family
$(p_j: \prod M_i \to M_j | j \in I)$ of homomorphisms is a 
{\em product} of the $M_i$, if for every $R$-semimodule $N$ 
and every family of homomorphisms $(f_j : N \to M_j | j \in I)$ 
there is a unique homomorphism $f: N \to \prod M_i$ such that 
 $$ \btriangle[N`\prod M_i`M_j;f`f_j`p_j] $$ 
commutes for all  $j \in I$.

 \item Let $(M_i | i \in I)$ be a family of $R$-semimodules. An 
$R$-semimodule $\coprod M_i$ together with a family
$(\iota_j: M_j \to \coprod M_i | j \in I)$ of homomorphisms is a 
{\em coproduct} of the $M_i$, if for every $R$-semimodule $N$ 
and every family of homomorphisms $(f_j : M_j \to N | j \in I)$ 
there is a unique homomorphism $f: \coprod M_i \to N$ such that 
  $$ \qtriangle[M_j`\coprod M_i`N;\iota_j`f_j`f] $$
commutes for all  $j \in I$.
 \end {enumerate}
 \end{defn}

 \begin{lma}
 Products and coproducts are unique up to isomorphism.
 \end{lma}

 \begin{proof} Straightforward.
 \end{proof}

 \begin {thm} {\em (Rules of computation in products of semimodules.)} 
 Let $(\prod M_i, (p_j))$ be a product of the family of $R$-semimodules 
$(M_i)_{i \in I}$. Let $(a_i \in M_i | i \in I)$ be a family of elements. 
Then there is a unique $a \in \prod M_i$, such that $p_i (a) = a_i$ for 
all $i \in I$. If $(b_i \in M_i| i \in I)$ is another family with $b \in 
\prod M_i$ and $p_i (b) = b_i$, then $a+b$ is the unique element in 
$\prod M_i$ such that $p_i (a+b) = a_i+b_i$ for all $i \in I$. Furthermore 
$ra$ is the unique element in $\prod M_i$ such that $p_i (ra) = ra_i$ 
for all $i \in I$.
 
 For each $a \in \prod M_i$ there is a unique family 
$(a_i |i \in I)$ with $p_i(a) = a_i.$ 
 \end {thm}

 \begin {proof}
 Given a family $(a_i \in M_i| i \in I)$. We define $\varphi_i : \{1\} 
\to M_i$ by $\varphi_i (1) = a_i$ for all 
$i \in I$. Given $g_i \in \Hom_R (R,M_i)$, such that the diagrams 
 $$ \qtriangle[\{1\}`R`M_i;`\varphi_i`g_i] $$
 commute.
 Then there is a unique $g: R \to \prod M_i$ such that
 $$ \btriangle[R`\prod M_i`M_j;g`g_j`p_j] $$ 
 commutes for all $i \in I$. So we get a commutative diagram 
 $$\bfig 
 \putmorphism(0, 1000)(1, 0)[\{1\}`R`]{500}1a
 \putmorphism(0, 1000)(1, -1)[`\prod M_i`]{500}1l
 \putmorphism(0, 1000)(1, -2)[`M_j.`\varphi_j]{500}1l
 \putmorphism(500, 1000)(0, -1)[``g]{500}1r
 \putmorphism(500, 500)(0, -1)[``p_j]{500}1r
 \putmorphism(500, 1000)(1, -1)[``]{260}1l
 \putmorphism(700, 870)(0, -1)[``g_j]{740}1r
 \putmorphism(760, 250)(-1, -1)[``]{260}1l
 \efig$$
 The homomorphism $g$ is uniquely determined by $g(1)$. 
We define $a := g(1)$ and get $p_j(a) = \varphi_j (1) = a_j$.
 
 Since $a$ is uniquely determined by the $p_j(a) = a_j$, we get  
$p_j (a+b) = p_j (a) + p_j(b) = a_i + b_j$. The last statement is then clear.
 \end {proof}

 \begin {thm} 
 {\em (Rules of computation for coproducts of semimodules.)} Let $(\coprod M_i, (\iota_j))$ 
be a coproduct of the family of $R$-semimodules $(M_i)_{i \in I}$. Then the 
homomorphisms $\iota_i: M_i \to \coprod M_j$ are injective. For each element 
$a \in \coprod M_j$ there is a finite subset $J \subseteq I$ and a finite family
$(a_i \in M_i | i \in J)$ with $a= \sum_{i \in J} \iota_i (a_i)$. 
The $a_i \in M_i$ are uniquely determined by $a$. 
 \end {thm}

 \begin{proof} For a given $k$ define homomorphisms $f_i:M_i \to M_k$ by 
 $$f_i = \left\{\begin{array}{ll}
 \id, &i=k \\
 0,& \text{else. }
 \end{array}\right.$$
 Then
 $$ \qtriangle[M_i`\coprod M_j`M_k;\iota_i`f_i`f] $$ 
 defines a unique homomorphism $f$. For $i=k$ we get $f \iota_k = \id_{M_k}$, hence 
$\iota_k$ is injective.
 
 Define $\widetilde M := \sum \iota_j(M_j) \subseteq \coprod M_j$. Let 
$\kappa_i: M_i \to \widetilde M$ and $\kappa : \widetilde M \to \coprod M_j$ be 
the inclusion homomorphisms. Then $\kappa \kappa_i = \iota_i$ for all $i$. Furthermore there is a 
unique homomorphism $p: \coprod M_j \to \widetilde M$ such that the diagrams
 $$ \qtriangle[M_i`\coprod M_j`\widetilde M;\iota_i`\kappa_i`p] $$ 
 commute. The commutative diagram 
 $$\bfig 
 \putmorphism(0, 1000)(1, 0)[M_i`\coprod M_j`\iota_i]{500}1a
 \putmorphism(0, 1000)(1, -1)[`\widetilde M`\kappa_i]{500}1r
 \putmorphism(0, 1000)(1, -2)[`\coprod M_j.`\iota_i]{500}1l
 \putmorphism(500, 1000)(0, -1)[``p]{500}1r
 \putmorphism(500, 500)(0, -1)[``\kappa]{500}1r
 \putmorphism(500, 1000)(1, -1)[``]{260}1l
 \putmorphism(700, 870)(0, -1)[``\id]{740}1r
 \putmorphism(760, 250)(-1, -1)[``]{260}1l
 \efig$$
 implies $\id = \kappa p$ by the uniqueness property, so $\kappa$ is also 
surjective. Thus $\widetilde M = \coprod M_j$.

 Let $a = \sum \iota_j (a_j)$. Form $f$ as in the beginning of the proof. 
Then we have $f(a) = f (\sum \iota_j (a_j)) = \sum f 
\iota_j (a_j) = \sum f_j (a_j) = a_i$, so the $a_i$ 
are uniquely determined by $a$. 
 \end{proof}

 \begin{thm} 
 \begin{enumerate}
 \item $R\xsMod$ has  (direct) products. 
 \item $R\xsMod$ has coproducts (or direct sums).
 \end{enumerate}
 \end{thm}

 \begin {proof}
 1. Define $\prod M_i:= \{\alpha : I \to 
\cup_{i \in I} M_i |\ \forall j \in I : \alpha (j) \in 
M_j\}$ and $p_j: \prod M_i \to M_j$, $ p_j 
(\alpha) = \alpha (j) \in M_j$. It is easy to see, 
that $\prod M_i$ is an $R$-semimodule with componentwise 
operations and that the $p_j$ are homomorphisms.
If $(f_j: A \to M_j | j \in I)$ is a family of homomorphisms then 
there is a unique map $f: A \to \prod M_i$ such that 
$ p_j f = f_j$ for all $ j \in I$. The following families are identical: 
$(p_j f (a+b)) = (f_j (a+b)) = (f_j(a)+f_j 
(b))=(p_j f (a) + p_j f (b)) = (p_j (f(a)+f(b)))$ hence
$f (a+b)= f (a) + f (b)$. Analogously we see $f(r a ) = r b 
(a)$. Thus $f$ is a homomorphism and $(\prod 
M_i, (p_j))$ is a product of semimodules.

 2. Define $\coprod M_i: = \{\beta:I \to 
\cup_{i\in I} M_i |\ \forall j \in I : \beta (j) \in M_j, \beta 
\mbox { has finite}$ $\mbox{support}\}$ and $\iota_j:M_j \to \coprod 
M_i$, $ \iota_j (m_j) (i) = \delta_{ij} m_i$. Then 
$\coprod M_i \subseteq \prod M_i$ is a subsemimodule and 
$\iota_j$ are homomorphisms. Given $(f_j : M_j 
\to A| j \in I)$. Define $f(m) = f (\sum 
\iota_i m_i)= \sum f \iota_i (m_i) = \sum f_i (m_i)$. Then 
$f$ is a homomorphism of $R$-semimodules and we get $ f 
\iota_i (m_i) = f_i (m_i)$. Hence $f \iota_i = f_i$. If $g 
\iota_i$ for all  $i \in I$, then $g (m) = g (\sum 
\iota_i m_i) = \sum g \iota_i (m_i) = \sum f_i (m_i)$, thus 
$f=g$. 
 \end{proof}

 \begin{thm} \label{3.9}
 Let $(M_i|i \in I)$ be a family of subsemimodules of $M$. 
Then the following are equivalent:
 \begin{enumerate}
 \item $(M, (\iota_i: M_i \to M | i \in I))$ is a coproduct in $R\xsMod$.
 \item $M = \sum_{i \in I} M_i$ and $(\sum m_i = \sum m'_i$ $ 
\Longrightarrow $ $\forall i \in I: m_i = m'_i)$. 
 \end{enumerate}
 \end{thm}

 \begin{defn}
 If one of the equivalent conditions of Theorem \ref{3.9} 
is satisfied, then $M$ is called an {\em internal direct sum} of the  
$M_i$, and we write $M = \oplus_{i \in I} M_i$. 
 \end{defn}

 \begin{proof} 
 $(1) \Longrightarrow (2):$ We have an inclusion $\iota: \sum_{i \in I} M_i \to M$. 
Let $\kappa_i: M_i \to \sum_{i \in I} M_i$ 
be the inclusion homomorphisms. These homomorphisms induce a homomorphism 
$\kappa: M \to \sum_{i \in I} M_i$ such that the following diagram
  $$\bfig 
 \putmorphism(0, 1000)(1, 0)[M_j`M`\iota_j]{500}1a
 \putmorphism(0, 1000)(1, -1)[`\sum_{i \in I} M_i`\kappa_j]{500}1r
 \putmorphism(0, 1000)(1, -2)[`M`\iota_j]{500}1l
 \putmorphism(500, 1000)(0, -1)[``\kappa]{500}1r
 \putmorphism(500, 500)(0, -1)[``\iota]{500}1r
 \putmorphism(500, 1000)(1, -1)[``]{260}1l
 \putmorphism(700, 870)(0, -1)[``\id_M]{740}1r
 \putmorphism(760, 250)(-1, -1)[``]{260}1l
 \efig$$
commutes. This implies $\id = \iota \kappa$ by the uniqueness property, so $\iota$ is  
surjective. Thus $M = \sum_{i \in I} M_i$.

 If $\sum m_i = \sum m'_i$, then use the diagrams
  $$ \qtriangle[M_j`M`M_k;\iota_j`\delta_{jk}`p_k] $$
 to get $m_k = \sum_j \delta_{jk} (m_j) = \sum_j p_k 
\iota_j (m_j) = p_k (\sum_j m_j) = 
p_k (\sum_j m'_j) = \sum_j p_k 
\iota_j (m'_j) = \sum_j \delta_{jk} (m'_j) = m'_k$. 

$(2) \Longrightarrow (1):$ 
 Define $f$ in the diagram
 $$ \qtriangle[M_i`M`N;\iota_i`f_i`f] $$ 
 by $f (\sum m_i) : = \sum f_i (m_i)$. By (2) $f$ is a 
well-defined homomorphism and we have $f \iota_j 
(m_j)= f_j (m_j) = f (m_j)$. Furthermore $f$ 
is uniquely determined since $g \iota_j = f_j$ $\Longrightarrow$ $g 
(\sum m_i) = \sum g (m_i) = \sum g \iota_i (m_i) = \sum f_i 
(m_i) = f (\sum m_i) \Longrightarrow f = g$. 
 \end{proof}

 \begin{rem}
 If $(M, (\iota_i: M_i \to M | i \in I))$ is an internal direct sum then
 $$M = \sum_{i \in I} M_i \mbox{ and } \forall i \in I: M_i 
\cap \sum_{j\not=i, j \in I} M_j = 0.$$ This is an immediate consequence of 
Theorem \ref{3.9} (2). Another consequence of (2) is  
$$M = \sum_{i \in I} M_i \mbox{ and } (\sum m_i = 0 \Longrightarrow \forall i \in I: m_i = 0).$$ 
 However, neither of these two conditions implies, that $(M, (\iota_i: M_i \to M | i \in I))$ 
is an internal direct sum, as the following example shows.
 $M := \{0, 1_A, 1_B, 2_B\}$ with the addition
 $$\begin{tabular}{||c||c|c|c|c||} \hline \hline
 +        & $0$ & $1_A$ & $1_B$ & $2_B$ \\ \hline \hline
 $0$ & $0$ & $1_A$ & $1_B$ & $2_B$\\ \hline 
 $1_A$ & $1_A$ & $1_A$ & $2_B$ & $2_B$\\ \hline 
 $1_B$ & $1_B$ & $2_B$ & $2_B$ & $2_B$\\ \hline 
 $2_B$ & $2_B$ & $2_B$ & $2_B$ & $2_B$\\ \hline \hline
 \end{tabular}$$
 has the subsemimodules $M_A:= \{0, 1_A\}$ and $M_B:= \{0, 1_B, 2_B\}$ satisfying
$M = M_A + M_B$ and $M_A \cap M_B = 0$. It also satisfies $m_1 + m_2 = 0 \impl 
m_1 = m_2 = 0$. Because of $1_A + 1_B = 0 + 2_B$ it does not, however, satisfy 
$m_1 + m_2 = m'_1 + m'_2 \impl m_1 = m'_1, m_2 = m'_2$.
 \end{rem}
 
 \begin{thm} 
 Let $(\coprod M_i, (\iota_j: M_j \to 
\coprod_{i \not= j} M_i))$ be a coproduct in $R\xsMod$. Then  
$\coprod M_i$ is an internal direct sum of the $\iota_j (M_j)$. 
 \end{thm}

 \begin{proof}
 $\iota_j$ injective $\impl M_j \cong \iota_j 
(M_j) \impl$ 
 $$\bfig 
 \putmorphism(0, 500)(1, 0)[M_j \iso \iota_j(M_j)`\coprod M_i`]{750}1a
 \putmorphism(-250, 500)(2,-1)[``]{1000}1l
 \putmorphism(250, 500)(1, -1)[``]{500}1l
 \putmorphism(750, 500)(0, -1)[`N`]{500}1r
 \efig$$
 This defines a coproduct. By \ref{3.9} we have a internal direct sum. 
 \end{proof}

 \begin{defn} 
 A subsemimodule $M \subseteq N$ is called a {\em direct 
summand} of $N$, if there is a subsemimodule $M' \subseteq N$ 
such that $N = M \oplus M'$ as an internal direct sum. 
 \end{defn}

 \begin{thm} \label{3.13}
 Let $M \subseteq N$ be a direct summand of $N$. Then
 the following hold:
 \begin{enumerate}
 \item there is a $p \in R\xsMod(N, M)$ such that 
 $$\bfig 
 \putmorphism(0, 0)(1, 0)[(M`N`\iota_{ }]{500}1a
 \putmorphism(500, 0)(1, 0)[\phantom{N}`M) \ = \ \id_M,`p]{750}1a
 \efig$$
 \item there is an $f \in R\xsMod (N, N)$ such that $f^2 = f$ 
and $f(N)=M.$ 
 \end{enumerate}
 \end{thm}

 \begin{proof} 
 (1) Let $M_1 : = M$ and $M_2 \subseteq N$ with $N = M_1 \oplus M_2$. 
 We define $p=p_1 : N \to M_1$ by 
 $$ \qtriangle[M_i`N`M_1;\iota_i`\delta_{i1}`p_1] $$
 where $\delta_{ij} = 0$ for $i \not= j$ and $\delta_{ij} 
= \id_{M_i}$ for $i=j$. Then we have $p_1 \iota_1 = 
\delta_{11} = \id_M$. 

 (2) For $f: = \iota p : N 
\to N$ we have $f^2 = \iota p \iota p = \iota p = 
f$ because of $p \iota =  \id$. Furthermore $f (N) = \iota p (N)= 
M$, since $p$ is surjective. 

 \end{proof}

 \begin{rem}
 Neither of the two conditions implies, that $M$ is a direct summand of $N$,
as the following example shows.
 $N := \{0, 1, 2\}$ with the addition
 $$\begin{tabular}{||c||c|c|c||} \hline \hline
 +        & $0$ & $1$ & $2$ \\ \hline \hline
 $0$ & $0$ & $1$ & $2$\\ \hline 
 $1$ & $1$ & $1$ & $2$\\ \hline 
 $2$ & $2$ & $2$ & $2$\\ \hline \hline
 \end{tabular}$$
 has the subsemimodules $M_1:= \{0, 1\}$ and $M_2:= \{0, 2\}$. $N$ is not the direct 
sum of the subsemimodules $M_1$ and $M_2$, but $p_1: M \to M_1$ defined by $p_1(1) = 1$, 
$p_1(2) = 1$ satisfies $p_1 \iota_1 = \id_{M_1}$. And $f := \iota_1 p_1$ satisfies condition (2).
 \end{rem}

 \section{The Structure of Semiideals of $\N_0$} 
 
 It is well known, that any ideal of $\mathbb Z$ is cyclic, i.e. a principal ideal. For the semiring $\N_0$ the structure of semiideals is somewhat more complicated and will be studied in this section. We will call a semiideal $M \subseteq \N_0$ {\em cyclic} if there is an element $d \in M$ such that $M = \N_0 \cdot d$. The set $\{4 + n \cdot 2 | n \in \N_0\} \cup \{0\}$ is clearly a semiideal, but not cyclic. Another example of a semiideal, that is not cyclic, is $\{8 + n \cdot 2 | n \in \N_0\} \cup \{0, 4\}$.
 
 \begin{defn}
 Let $M \subseteq \N_0$, $M \not= 0$ be a semiideal. 
 \begin{enumerate}
 \item An element $d \in \N$ is called a {\em difference} of $M$, if there exist two elements $a, b \in M, a \not= 0$ such that $a + d = b$.
 \item The minimal difference $d$ of $M$ is called the {\em period} of $M$.
  \hfill $\Box$
 \end{enumerate}
 \end{defn}
 
 Clearly any element $a \not= 0$ in $M$ is a difference of $M$. However, a difference of $M$ is not necessarily an element of $M$.
 
 \begin{thm} \label{period}
 Let $M \subseteq \N_0$ be a semiideal. If $d$ is a difference of $M$, then there is an element $c \in M, c \not= 0$ divisible by $d$, such that $c + nd \in M$ for all $n \in \N_0$.
 \end{thm}
 
 \begin{proof}
 Let $a + d = b, a \not= 0$. Let $n = qb + r$ with $q, r \in \N_0$ and $r < b$. Then there is an $s \in \N_0$ with $r + s = b$. Set $c' := ab$.  Then $c' + nd = ba + rd + qdb = sa + ra + rd + qdb = sa + r(a + d) + qdb = sa + rb + qdb = sa + (r + qd)b \in M$, thus $c' + nd \in M$ for all $n \in \N_0$. Set $c := dc'$ to get $c + nd \in M$ for all $n \in \N_0$ and $c$ divisible by $d$.
 \end{proof}
 
 This Theorem immediately implies Lemmas 4, 5, 7, 17, 18, and Theorem 20 in \cite{AD} since $\{c+nd | n \in \N_0\} \cup \{0\} = dT_{c/d}$ in \cite{AD}. 
 
 \begin{lma} \label{period2}
 Let $d$ be the period of $M$ with $a + d = b, a \not= 0, a, b \in M$. Then $d$ divides $a$.
 \end{lma}
 
 \begin{proof}
  Let $a = qd + r$ with $q, r \in \N_0$ and $r < d$. If $r = 0$ then $a = qd$ and we are done. Assume $r\not=0$. Define $a' := (q + 1)a + qd = a + q(a + d) \in M$. Then $a' + r = (q + 1)a + qd + r = (q + 2) a \in M$, a contradiction to the minimality of $d$. Thus $r = 0$ and $a = qd$. 
 \end{proof}
 
 \begin{defn}
 A semiideal $M$ of $\N_0$ for which there is a $c \in M, c \not= 0$ and a $d \in \N = \N_0 \setminus \{0\}$ such that\\
 $\phantom{XXXXXXX} M = \{ c + nd | n \in \N_0 \} \cup \{0\}$\\
 is called a {\em periodic semiideal}. By \ref{period2} the element $d$ divides $c$.
  \hfill $\Box$
 \end{defn}
 
 \begin{lma}\label{period3}
 Let $M \not= 0$ be a semiideal of $\N_0$ and $d$ be the period of $M$. Then there is an element $c \in M, c \not= 0$ such that the elements of the form $c + nd$ are precisely the elements $c' \in M$ with $c \leq c'$. In particular there is a minimal such $c \in M$. 
 \end{lma}
 
 \begin{proof}
 By Theorem \ref{period} and Lemma \ref{period2} the elements of the form $c' := c + nd$ are in $M$ and clearly satisfy $c \leq c'$. Conversely let $c, c' \in M$ with $c < c'$. Then there is an $n \in \N$ with $c + nd < c' \leq c + (n + 1)d = c + nd + d \in M$. Since $c' \in M$ and $d$ has the minimality property we get $c' = c + (n + 1)d$. Hence $\{c+nd | n \in \N_0\} = \{c' \in M | c \leq c'\}$.
 \end{proof}
 
 \begin{defn}
 \begin{enumerate}
 \item The minimal $c$, as constructed in Lemma \ref{period3}, is called the {\em footing} of $M$.
 \item For a semiideal $M \not= 0$ of $\N_0$, the semiideal\\
  \phantom{XXXXX} $\mbox{perc}(M) := \{ c + nd | n \in \N_0 \} \cup \{0\}$\\
  where $d$ is the period and $c$ is the footing of $M$, is called the {\em periodic core} of $M$. 
 \hfill $\Box$
 \end{enumerate}
 \end{defn}
 
 \begin{lma}\label{period4}
 Let $d \in \N_0, d \geq 1$ be the period of $M$. Then $d$ divides any element $c \in M$. Furthermore $d$ is the greatest common divisor of all elements of $M$ or of any generating set of $M$.
 \end{lma}

 \begin{proof} 
 Given $c \in M$. Let $\gcd(c,d) = t$. Then it is well-known from elementary number theory that $nc = n'd + t$ for some $n, n' \in \mathbb Z$. If $n, n'$ are both negative, then add multiples of $cd$ to both sides of the equation to get $nc = n'd + t$ for some $n, n' \in \N_0$.
 
  Let $c', c' + d \in M$. Then $n'(c' + d) \in M$. Furthermore $n'(c' + d) + t = n'c' + n'd + t = n'c' + nc \in M$. So $t$ is a difference of $M$. Since $d$ is the period of $M$ and $t \leq d$, we get $t = d$ and thus $d$ divides $c$. 
  
  By Theorem \ref{period} we get that there is a $c'' \in M, c'' \not= 0$ such that $c'' + nd \in M$ for all $n \in N_0$. Thus $d$ is the greatest common divisor of all elements of $M$ and clearly also of all elements of any generating set of $M$.
 \end{proof}
 
 \begin{thm} \label{period5}
 Let $M \subseteq \N_0$, $M \not= 0$ be a semiideal. Let $d$ be the period of $M$ and $c$ be the footing of $M$. Then 
 \begin{enumerate}
 \item $d$ is the greatest common divisor of all elements of $M$,
 \item $\mbox{\rm perc}(M) = \{ c + nd | n \in \N_0 \} \cup \{0\}  \subseteq M$ is a subsemiideal of $M$,
 \item $\{ c + nd | n \in \N_0 \} = \{c' \in M | c \leq c'\}$,
 \item $c - d \notin M$.
 \end{enumerate}
 \end{thm} 
 
 Note with respect to (4) of the Theorem, that all elements $c - e \notin M$ for all $1 \leq e \leq 2d - 1$. 
 
 The following Theorem implies Theorems 9 and 22 in \cite{AD} and Theorem 4 in \cite{GC}.
 
 \begin{thm}
 Let $M \subseteq \N_0$, $M \not= 0$ be a semiideal. Then there is a finite unique smallest (canonical) generating system $X$ for $M$. It is contained in any other generating system of $M$. Its cardinality is less than or equal $e/d$ where $d$ is the period of $M$ and $e$ is the smallest element $\not= 0$ of $M$. Furthermore $\N_0$ is Noetherian. 
 \end{thm}
 
 \begin{proof}
 We give an algorithm to construct $X$. By $\langle a, b, c, \ldots \rangle$ we denote the semiideal generated by the elements $a, b, c, \ldots$ 
 
 Let $e_0 := e$ be the smallest element $\not= 0$ in $M$. Clearly it cannot be written as a linear combination of any other elements in $M$, which are larger than $e_0$. So $e_0$ must be contained in any generating system $Y$ of $M$.
 
 Let $e_1$ be the smallest element $\notin \langle e_0 \rangle$ in $M$. If such an element exists, it is not a multiple of $e_0$ and cannot be  
written as a linear combination with additional elements in $M$, which are larger than $e_1$. So $e_1$ must be contained in any generating system $Y$ of $M$.
 
 Assume that $e_0, \ldots, e_{n-1}$ have been constructed by the algorithm.
 
 Let $e_n$ be the smallest element $\notin \langle e_0,  \ldots, e_{n-1} \rangle$ in $M$. If such an element exists, it is not a linear combination of the $e_0, \ldots, e_{n-1}$ and cannot be  
written as a linear combination with additional elements in $M$, which are larger than $e_n$. So $e_n$ must be contained in any generating system $Y$ of $M$.

 This algorithm stops after at most $e/d$ steps. Indeed, assume $e_i \equiv e_j \mod e_0$ with $e_j > e_i$. Then there are $q_i, q_j, r \in \N_0$ with $e_i = q_ie_0 + r$, $e_j = q_j e_0 + r$ and $q_j > q_i$. All $e_0, \ldots, e_n$ are multiples of $d$, the period of $M$, thus $r$ is also a multiple of $d$. We get $e_j = q_i e_0 + r + (q_j - q_i) e_0 = e_i +  (q_j - q_i) e_0 \in \langle e_0, \ldots, e_i \rangle$. So any two elements of $\{ e_0, \ldots, e_n \}$ most be mutually incongruent w.r.t. $e_0$. There are only $e/d$ congruence classes modulo $e_0$ with representatives $r_i$ a multiple of $d$, so $n < e/d$.
 
 Furthermore observe that for an arbitrary semimodule $M$ the following conditions are equivalent:
 \begin{itemize}
 \item Each subsemimodule of $M$ is finitely generated;
 \item $M$ satisfies the ACC;
 \item $M$ satisfies the maximum condition.
 \end{itemize}
 The standard proof for modules \cite{HT} works as well for semimodules. Hence $\N_0$ is Noetherian.
 \end{proof}
 
 In contrast to the ideals in $\mathbb Z$ the semiideals $M$ of $\N_0$ are not cyclic (generated by one element).

 Obviously there are infinitely many periodic semiideals in $M$ of period $d$ by taking any element $c \in \mbox{\rm perc}(M), c \not= 0$ and forming sets $\{ c + nd | n \in \N_0 \} \cup \{0\}$, where $d$ is the period of $M$. Furthermore if $d$ divides $d' \in M$ then there are semiideals $\{ c' + nd' | n \in \N_0 \} \cup \{0\}$ where $c' \in \mbox{\rm perc}(M)$ is divisible by $d'$.
 
 If $M = \N_0 \cdot d$ is cyclic, then the period of $M$ is $d$. The footing of $M$ is $d$, too.
 
 It is easy to determine the period $d$ of a semiideal $M$. It is the greatest common divisor of a finite generating set of $M$. 
 
 \begin{rem}
 Given $a, b \in \N_0$ with $\gcd(a,b)=d$. Let $M := \langle a, b \rangle$. Then by Theorem \ref{period5} (1) and Theorem \ref{period} we get $dT_n \subseteq M$ as in \cite{AD} Theorem 19.
 \end{rem}
 
 On the way of determining the footing of a semiideal $M$ with two generators we encounter an interesting number-theoretic observation.

 \begin{thm} \label{period6}
 Let $a, b$ be natural numbers with $a, b > 0$. Let $d > 1$ be a natural number.
 \begin{enumerate}
 \item $a$ and $b$ are relatively prime if and only if there are $r, s \in \N_0$ such that
     $$(a - 1)(b - 1) = r a + s b.$$
 \item Let $a, b$ be divisible by $d$. Then $\gcd(a,b) = d$ if and only if there are $r, s \in \N_0$ such that
     $$d(a/d - 1)(b/d - 1) = r a + s b.$$
 \end{enumerate}
 \end{thm}
 
 \begin{proof}
 (1) Assume $(a - 1)(b - 1) = r a + s b$. If $t > 1$ is a common divisor of $a$ and $b$, then $r a + s b$ is divisible by $t$, so $ab - a - b + 1 = (a - 1)(b - 1)$ is divisible by $t$, hence 1 is divisible by $t$, a contradiction. So $a, b$ are relatively prime.
 
 Let $a, b$ be relatively prime. Then w.l.g. there are $u > 0$ and $v \geq 0$ such that $\gcd(a,b) = 1 = ua - vb$. We may also assume $a > v$. Indeed if $v \geq a$ then write $v = qa + v'$ with $a > v' \geq 0$. We get 
 $$1 = ua - vb = ua - qab - v'b = (u - qb)a - v'b = u'a - v'b$$ with $u' = u - qb$ and $a > v' \geq 0$. Furthermore $u'$ must be positive.
 
 Set $r := u - 1$ and $s := a - v - 1$. Then $r \geq 0$ and $s \geq 0$. We get
 $ra + sb = (u-1)a + (a-v-1)b = ua - a + ab - vb - b = ab - a - b + ua - vb = ab - a - b + 1 = (a-1)(b-1)$.
 
 (2) Let $a = da'$ and $b = db'$. Then $\gcd(a,b) = d \iff \gcd(a',b') = 1 \iff \exists r,s \in \N_0: (a' - 1)(b' - 1) = ra' + sb' \iff \exists r,s \in \N_0: d(a' - 1)(b' - 1) = d(ra' + sb') = ra + sb$.
 \end{proof}
 
 \begin{thm}
 Let $M$ be a semiideal of $\N_0$ generated by two elements $a, b \in M$. Let $\gcd(a,b) = d > 0$. Then the footing of $M$ is $c = d(a/d - 1)(b/d - 1)$.
 \end{thm}
 
 \begin{proof}
 Let $d$ be the period of $M$. Each element of $M$ is a multiple of $d$. Then $M' := \{a/d \ |\ a \in M \}$ is a semiideal with period $1$ and generators $a/d$ and $b/d$. Multiplication by $d$ is an isomorphism between $M'$ and $M$. 
 
 In particular the footing $c$ of $M$ is $d$-times the footing $c'$ of $M'$. So we only have to show $c' = (a' - 1)(b' - 1)$ is the footing of $M'$ and we will get $c = d \cdot c' = d \cdot (a/d - 1)(b/d - 1)$.
 
 Let $M$ be generated by $a$ and $b$ and let $d = \gcd(a,b) = 1$ and $c := (a - 1)(b - 1)$. Let $1 = ua - vb$ with $u > 0, a > v \geq 0$.  To show that $c$ is the footing of $M$, we have to show, that $c - 1 \notin M$ and $c + t \in M$ for all $t \in \N_0$.
 
 Assume $c - 1 \in M$. Then $(a - 1)(b - 1) - 1 = ra + sb$ with $r,s \in \N_0$. Thus $ab - a - b = ra + sb$ with $r, s \geq 0$ or $ab = (r + 1)a + (s + 1)b$. Since $\gcd(a,b) = 1$ and $ab$ and $(r + 1)a$ are divisible by $a$, so are $(s + 1)b$ and thus $s + 1$ with $s + 1 > 1$. So we get $s + 1 = ta$ and $ab = (r + 1)a + tab$ with $t > 0$. Since we are in $\N_0$ we must have $t = 1$ and $r + 1 = 0$ a contradiction to $r \geq 0$. Thus $c - 1 \notin M$. 
 
 The case $t = 0$ for $c + t \in M$ has been handled in \ref{period6} (1). So we can assume $t > 0$. By the Euclidean algorithm write $tv + 1 = (q + 1)a - h$ with $0 \leq h < a$. Define $r := tu - 1 - qb$, and $s := (q + 1)a - tv - 1$. Then $ra + sb = (tu - 1 - qb)a + ((q + 1)a - tv - 1)b = tua - a - qab + qab + ab - tvb - b = ab - a - b + t(ua - vb) = (a - 1)(b - 1) + (t - 1)$ for all $t > 0$.
 
 We have to show that $r \geq 0$ and $s \geq 0$. We have $h < a \impl (q + 1)a - (tv + 1) < a \impl qa - 1 < tv \impl tv \geq qa \impl tua = tvb + t \geq qab + t \impl tu \geq qb + t/a > qb \impl r = tu - qb - 1 > qb - qb - 1 = -1 \impl r \geq 0$. Furthermore we have $ s = (q + 1)a - (tv + 1) = (q + 1)a - (q + 1)a + h = h \geq 0$.
  \end{proof}
 
  Now we want to determine quotients of $\N_0$ by a semiideal $M$.
 
 \begin{thm}
 Let $M \not= 0$ be a semiideal of $\N_0$ of period $d$. Then $\N_0/M \iso \mathbb Z/(d)$.
 \end{thm} 

 \begin{proof}
 Consider the natural homomorphism $f: \N_0 \to \mathbb Z/(d)$. Let $n, n' \in \N_0$ be congruent in the congruence relation of Proposition \ref{congrel4} for $\N_0/M$. Then there are $a, b \in M$ with $n + a = n' + b$. Since $a, b$ are divisible by $d$, we get $f(n) = f(n')$, so we get a homomorphism $g: \N_0/M \to \mathbb Z/(d)$. If $n$ is divisible by $d$, then there is an $r \in \N_0$ such that $rd \in \mbox{\rm perc}(M)$ and $n + rd \in \mbox{\rm perc}(M)$. Hence $n + rd = 0 + (n + rd)$ and thus $n \sim_M 0$. This shows that $g: \N_0/M \to \mathbb Z/(d)$ is an isomorphism.
 \end{proof}

 With Example \ref{congrrel9} we see that there are also quotient semimodules of $\N_0$ that are not obtained by factoring out a subsemimodule.
 
  Let $M$ be a semiideal in $\N_0$. Then $\N_0/M=0$ iff there are $a,b$ in $M$ with $a=b+1$. There are infinitely many such subsemimodules in $\N_0$.

 \section{Free Semimodules}
   \begin{defn}
 Let  $R$ be a semiring. The {\em underlying set functor} $V: R\xsMod \to \Set$ is given for $M$ and $f: M' \to M$ in $R\xsMod$ by\\
 \phantom{XXXXX} $V(M):= |M|$ and $V(f) := |f|$,\\
 where $|.|$ stands for the underlying set. It satisfies\\
 \phantom{XXXXX} $V(f \circ g) = V(f) \circ V(g)$.  \hfill $\Box$
 \end{defn}

 \begin{defn}
 Let $X$ be a set and $R$ be a semiring. An $R$-semimodule $RX$ 
together with a map $\iota : X \to RX$ is called a {\em 
free} $R$-{\em semimodule generated by} $X$ (or {\em an 
$R$-semimodule freely generated by $X$}), if for every 
$R$-semimodule $M$ and for every map $f: X \to M$  there 
exists a unique homomorphism of $R$-semimodules $g: RX \to 
M$ such that the diagram 
 $$\qtriangle[X`V(RX)`V(M);\iota`f`V(g)]$$
 commutes or more simply $g\circ \iota = f$. 

 An $R$-semimodule $F$ is a {\em free $R$-semimodule} if there is a 
set $X$ and a map $\iota: X \to F$ such that $F$ is freely 
generated by $X$. Such a set $X$ (or its image $\iota(X)$) 
is called a {\em free generating set for $F$}. 
 \hfill $\Box$
 \end{defn}

 {\em Remark:} We will see further down that the map $\iota : X \to RX$ is injective, so that we can regard $X$ as a subset of $RX$. The definition of a free semimodule then says, that in order to define a homomorphism $g: RX \to M$ with a free semimodule $RX$ as domain one needs only to define $g$ on the elements of $X$ (considered as a subset of $RX$) and that the images of the elements of $X$ may be chosen arbitrarily.

 \begin{prop}
 A free $R$-semimodule $\iota : X \to RX$ defined over a set 
$X$ is unique up to a unique isomorphism of $R$-semimodules. 
 \end{prop}

 \begin{proof} follows from the following diagram
 $$\bfig 
 \putmorphism(1000, 500)(-2, -1)[X``\iota]{1000}1l
 \putmorphism(1000, 500)(-1, -1)[``\iota']{500}1r
 \putmorphism(1000, 500)(1, -1)[``\iota]{500}1l
 \putmorphism(1000, 500)(2, -1)[``\iota']{1000}1r
 \putmorphism(0, 0)(1, 0)[RX`RX'`h]{500}1a
 \putmorphism(500, 0)(1, 0)[\phantom{RX'}`\phantom{RX}`k]{1000}1a
 \putmorphism(1500, 0)(1, 0)[RX`RX'`h]{500}1a
 \efig$$
 \end{proof}

 \begin{prop} \label{freerules}
 {\em (Rules of computation in a free $R$-semimodule)} 
 Let $\iota: X \to RX$ be a free $R$-semimodule over $X$. Let 
$\widetilde{x} : = \iota (x) \in RX$ for all  $x \in X$. 
Then we have 
 \begin{enumerate}
 \item $\widetilde{X} = \{\widetilde{x} |\ \exists x \in X : 
\widetilde{x} = \iota (x)\}$ is a generating set of $RX$, 
i.e. each element $m \in RX$ is a linear combination $m = 
\Sigma_{x \in X} r_x \widetilde{x}$ (finite sum). 
 \item $\iota: X \to RX$  is injective and $\widetilde{X} \subseteq RX$ 
is linearly independent, i.e. if $\Sigma_{x \in X} r_x 
\widetilde{x} = \Sigma_{x \in X} s_x 
\widetilde{x}$, then we have $\forall x \in X : r_x = s_x$. 
 \end{enumerate}
 \end{prop}

 \begin {proof} 
 (1) Let \\
\phantom{X}$B : = \langle \widetilde{x} | x \in X \rangle$ \\
\phantom{XXX}$= \{\Sigma_{x \in X} r_x \widetilde{x} | r_x \in R, \mbox{at most finitely many } r_x \not= 0\} 
\subseteq RX$ \\
denote the $R$-subsemimodule of $RX$ 
generated by the set of $\widetilde{x}$. Let 
$j: B \to FX$ be the embedding homomorphism. We 
get an induced map $\iota': X \to B$. The 
following diagram 
 $$\bfig
 \putmorphism(0, 500)(1, 0)[X`B `\iota']{500}1a
 \putmorphism(500, 500)(1, 0)[\phantom{B}`RX`j]{500}1a 
 \putmorphism(500, 0)(1, 0)[B `RX `j ]{500}1a
 \putmorphism(0, 500)(1, -1)[``\iota' ]{500}1l
 \putmorphism(1000, 500)(0, -1)[``jp]{500}1r
 \putmorphism(1000, 500)(-1, -1)[``p]{500}1r
 \efig$$
induces a unique $p$ with $p 
\circ j \circ \iota' = p \circ \iota = \iota'$. Because of $jp \circ 
\iota = j \circ \iota' = \iota = \id_{RX} 
\circ \iota$ we get $jp = \id_{RX}$, hence the 
embedding $j$ is surjective and thus the identity.  

 (2) Let $\Sigma_{x \in X} r_x \widetilde{x} = \Sigma_{x \in X} s_x \widetilde{x}$ with $r_0 
\not= s_0$ for some $x_0 \in X$. Let $j : X \to R$ be the map given by $j 
(x_0) = 1, j (x) = 0$ for all $x \not= x_0$. 
Then there exists a unique homomorphism $g: RX \to R$ with   
 $$ \qtriangle[X `RX `R ;\iota `j `g ] $$
 commutative and  
 $r_0 = \Sigma_{x \in X} 
r_x j (x) = \Sigma_{x \in X} r_x g (\widetilde x) = 
g (\Sigma_{x \in X} r_x \widetilde 
x) = g (\Sigma_{x \in X} s_x \widetilde 
x) = \Sigma_{x \in X} s_x g (\widetilde x) = \Sigma_{x \in X} 
s_x j (x) = s_0$. This is a contradiction. Now it is trivial to see 
that $\iota$ is injective. Hence the second statement. 
 \end{proof}

 \begin{notatn} Since $\iota$ is injective we will identify $X$ with it's 
image in $RX$ and we will write $\Sigma_{x \in X} r_xx$ for an element 
$\Sigma_{x \in X} r_x\iota(x) \in RX$. The coefficients $r_x$ are uniquely 
determined.
 \end{notatn}

 \begin{prop} \label{freeexist}
 Let $X$ be a set. Then there exists a free $R$-semimodule 
$\iota: X \to RX$ over $X$. 
 \end{prop}

 \begin {proof} 
 Obviously $RX: = \{\alpha : X \to R \ |\  \alpha (x) = 0 \mbox{ for almost all } 
x \in X \}$ is a subsemimodule of  $\Set (X,R)$ 
which is an $R$-semimodule by componentwise addition and 
multiplication. Define $\iota: X \to RX $ by $\iota (x)(y): 
= \delta_{xy}$. 

 Let $f: X \to M $ be an arbitrary map. Let $\alpha \in 
RX$. Define $g (\alpha) : = \Sigma_{x \in X} \alpha (x) \cdot 
f (x).$ Then $g$ is well defined, because we have  $\alpha 
(x) \not= 0$ for only finitely many $x \in X$. Furthermore 
$g$ is an $R$-semimodule homomorphism: $r g (\alpha) + sg 
(\beta) = r \Sigma \alpha (x)\cdot f (x) + s \Sigma \beta (x) 
\cdot f (x) = \Sigma (r \alpha (x) + s \beta (x)) \cdot f (x) 
= \Sigma (r \alpha + s \beta) (x) \cdot f (x) =g (r \alpha + 
s \beta)$. 

 Furthermore we have $ g \iota = f : g \iota (x) = \Sigma_{y 
\in X} \iota (x) (y) \cdot f (y) = \Sigma \delta_{xy} \cdot f 
(y) = f (x).$ For $ \alpha \in RX$ we have $\alpha = 
\Sigma_{x\in X} \alpha (x) \iota (x)$ since $\alpha (y) = 
\Sigma_{x\in X} \alpha (x) \iota (x) (y)$. In order to show that $g$ 
is uniquely determined by $f$, let $h \in \Hom_R (RX,M)$ be 
given with $h \iota = f$. Then $h (\alpha) = h (\Sigma \alpha 
(x) \iota (x))= \Sigma \alpha (x) h \iota (x)= \Sigma \alpha 
(x) f (x) = g (\alpha)$ hence $h = g.$ 
 \end {proof}

 \begin{rem}
 Let $\iota : X \to RX$ be a free semimodule. Let $f: X \to M$ 
be a map and $g: RX \to M$ be the induced $R$-semimodule 
homomorphism. Then 
 $$g(\Sigma_{x \in X} r_xx) = \Sigma_{x \in X} r_xf(x).$$
 \end{rem}

 \section{Tensor Products}
 \begin{defnrem}
 Let $M_R$ and $ {}_RN$ be $R$-semimodules, and let $A$ be a commutative monoid. A map 
$f: M \times N \to A$ is called $R${\em -balanced} if
 \begin{enumerate}
 \item $f(m + m', n) = f (m,n) + f (m', n),$
 \item $ f (m,n+n') = f (m,n) + f (m,n'),$
 \item $f(m r, n) = f (m, r n ),$
 \item $f(0, n) = 0 = f(m, 0)$
 \end{enumerate}
 for all $r \in R, \  m, m' \in M, \  n, n' \in N$.

 Let $\Bal_R (M,N; A)$ denote the set of all $R$-balanced
maps $f: M \times N\break \to A$. 

 $\Bal_R (M,N; A)$ is a commutative monoid with $(f+g) (m,n): = \break
f (m,n)+ g (m,n)$. 
 \end{defnrem}

 \begin{defn} \label{tensorproduct} 
 Let $M_R$ and $ {}_RN$ be $R$-semimodules. A commutative monoid $M 
\otimes_R N$ together with an $R$-balanced map 
 $$\otimes : M \times N \ni (m,n) \mapsto m \otimes n \in M 
\otimes_R N$$ 
 is called a {\em {tensor product of }} $M$ {\em {and }} $N$ 
{\em over} $R$ if for each commutative monoid $A$ and 
for each $R$-balanced map $f: M \times N 
\to A$ there exists a unique homomorphism of semigroups $g: M 
\otimes_R N \to A$ such that the diagram 
 $$ \qtriangle[ M \times N ` M \tensor_R N ` A; \tensor ` f 
` g ] $$ 
 commutes. The elements of $M \otimes_R N $ are called {\em 
tensors}, the elements of the form $m \otimes n$ (in the image 
of the map $\tensor$) are called {\em {decomposable tensors}}. 

 {\em Warning:} If you want to define a homomorphism $f: M 
\tensor_R N \to A$ with a tensor product as domain you {\em 
must} define it by giving an $R$-balanced map defined on 
$M \times N$. 
 \hfill $\Box$
 \end{defn}

 \begin{prop}
 A tensor product $(M \otimes_R N, \otimes )$ defined by
$M_R$ and $ {}_RN$ is unique up to a unique isomorphism. 
 \end{prop}

 \begin{proof} Let $(M \otimes_R N, \otimes)$ and $(M 
\boxtimes_R N, \boxtimes)$ be tensor products. Then the following 
commutative diagram 
 $$\bfig \putmorphism(1000, 500)(-2, -1)[M \times 
N``\tensor]{1000}1l 
 \putmorphism(1000, 500)(-1, -1)[``\boxtimes]{500}1r
 \putmorphism(1000, 500)(1, -1)[``\tensor]{500}1l
 \putmorphism(1000, 500)(2, -1)[``\boxtimes]{1000}1r
 \putmorphism(0, 0)(1, 0)[M \tensor_R N`M \boxtimes_R 
N`h]{500}1a 
 \putmorphism(500, 0)(1, 0)[\phantom{M \boxtimes_R 
N}`\phantom{M \tensor_R N}`k]{1000}1a 
 \putmorphism(1500, 0)(1, 0)[M \tensor_R N`M \boxtimes_R 
N`h]{500}1a 
 \efig$$
 implies $k = h^{-1}$ by the uniqueness of the factorization morphism.
 \end {proof}

 Because of this fact we will henceforth talk about {\em 
the} tensor product of $M$ and $N$ over $R$.

 \begin{prop}
 {\em (Rules of computation in a tensor product)} Let $(M 
\otimes_R N, \otimes)$ be the tensor product. Then we have 
for all $r \in R$, $m, m' \in M$, $n, n' \in N$ 
 \begin{enumerate} 
 \item $M \otimes_R N = \{ \Sigma_i m_i \otimes n_i \ |\  m_i 
\in M, n_i \in N\},$ 
 \item $(m+m') \otimes n = m \otimes n + m' \otimes n,$
 \item $m \otimes (n+n') = m \otimes n + m \otimes n',$
 \item $m r \otimes n = m \otimes r n $ (observe in 
particular, that $\tensor: M \times N \to M \tensor N$ is 
not injective in general),
 \item $0 \tensor n = 0 = m \tensor 0,$
 \item if $f: M \times N \to A$ is an $R$-balanced map and 
$g: M \tensor_R N \to A$ is the induced homomorphism, then 
 $$g(m \tensor n) = f(m,n).$$
 \end{enumerate}
 \end{prop}

 \begin{proof}
 (1) Let $B: = \langle m \otimes n \rangle \subseteq M 
\otimes_R N$ denote the submonoid of $M \otimes_R N$ 
generated by the decomposable tensors $m \tensor n$. Let 
$j: B \to M \tensor_R N$ be the embedding homomorphism. We 
get an induced map $\tensor': M \times N \to B$. The 
following diagram 
 $$\bfig
 \putmorphism(0, 500)(1, 0)[M \times N `B `\tensor' ]{500}1a
 \putmorphism(500, 500)(1, 0)[\phantom{B} `M \tensor_R N `j 
]{500}1a 
 \putmorphism(500, 0)(1, 0)[B `M \tensor_R N `j ]{500}1a
 \putmorphism(0, 500)(1, -1)[``\tensor' ]{500}1l
 \putmorphism(1000, 500)(0, -1)[``jp]{500}1r
 \putmorphism(1000, 500)(-1, -1)[``p]{500}1r
 \efig$$
induces a unique $p$ with $p 
\circ j \circ \tensor' = p \circ \tensor = \tensor'$ 
since $\tensor'$ is $R$-balanced. Because of $jp \circ 
\tensor = j \circ \tensor' = \tensor = \id_{M \tensor_R N} 
\circ \tensor$ we get $jp = \id_{M \tensor_R N}$, hence the 
embedding $j$ is surjective and thus the identity.  

 (2) $(m+m')\otimes n = \otimes (m+m',n) = \otimes 
(m,n)+\otimes (m',n) = m \otimes n+m' \otimes n$. 

 (3), (4), and (5) analogously.

 (6) is precisely the definition of the induced 
homomorphism. 
 \end{proof}

 \begin{prop}
 Given $R$-semimodules $M_R$ and ${}_RN$. Then there exists a 
tensor product $(M \otimes_R N, \otimes)$. 
 \end{prop}

 \begin{proof}
 Define $M \otimes_R N : = \N_0 ( M \times N ) / \simz$ where $\iota: M \times N \to \N_0 ( M \times N )$ is a free $\N_0$-semimodule 
over $M \times N$ (the free commutative monoid) and $\sim$ is the 
congruence relation generated by (intersection of all congruence relations containing)
 \begin{itemize}
 \item[] $\iota (m+m',n) \simz \iota (m,n) + \iota (m',n)$
 \item[] $\iota(m, n+n') \simz \iota (m,n) + \iota (m,n')$
 \item[] $\iota (m r,n) \simz  \iota (m, r n)$
 \item[] $\iota(0,n) = 0 = \iota(m,0)$
 \end{itemize}
 for all $r \in R$, $m, m' \in M$, $n, n' \in N$. Consider the diagram
 $$\bfig 
 \putmorphism(0, 500)(1, 0)[M \times N ` {\N_0} ( M 
\times N ) ` \iota]{750}1a 
 \putmorphism(750, 500)(1, 0)[\phantom{\N_0 ( M 
\times N )} `M \tensor_R N ` \nu]{750}1a 
 \putmorphism(1550, 500)(1, 0)[`= \N_0 ( M \times 
N ) /\simz`]{500}0a 
 \putmorphism(0, 500)(3, -1)[`A`\psi]{1500}1l
 \putmorphism(750, 500)(3, -2)[``\rho]{750}1r
 \putmorphism(1500, 500)(0, -1)[``g]{500}1r
 \efig$$  
 Let $\psi$ be $R$-balanced. Then there is a unique $\rho \in 
\Hom (\N_0 ( M \times N ), A)$ such that $\rho \iota 
= \psi$. So we get 
 $$\begin{array}{c}
 \rho (\iota(m+m',n)) = \psi (m+m',n) = \\
\psi(m,n) + \psi(m',n) = \rho(\iota (m,n)) + \rho(\iota (m', n)) = \\
 = \rho(\iota (m,n) +\iota (m', n)), \mbox{ (and similarly) }\\
 \rho (\iota(m,n + n')) = \rho(\iota (m,n) +\iota (m, n')),\\
 \rho (\iota(mr,n)) = \rho(\iota (m,rn)),\\
 \rho(\iota(0,n)) = 0 = \rho(\iota(m,0)).
 \end{array}$$ 
So by Theorem \ref{congrel2b} there is a unique $g 
\in \Hom (M \otimes_R N, A)$ such that $g \nu = \rho$. 

 Let $\otimes : = \nu \circ \iota$. 
Then $\otimes$ is balanced since $(m+m') \otimes n = \nu 
\circ \iota (m+m',n) = \nu(\iota 
(m+m',n)) = \nu( \iota (m,n)+ \iota (m',n)) = \nu \circ \iota (m,n) + \nu \circ \iota 
(m',n) = m \otimes n + m' \otimes n$. The other three properties are obtained in an analogous 
way. 

 We have to show that $(M \otimes_R N, \otimes)$ is a 
tensor product. The above diagram shows that for each 
commutative monoid $A$ and for each $R$-balanced map $\psi: M 
\times N \to A$ there is a $g \in \Hom (M \otimes_R N, A)$ 
such that $g \circ \otimes = \psi$. Given $ h \in \Hom (M 
\otimes_R N, A)$ with $h \circ \otimes = \psi$. Then $h 
\circ \nu \circ \iota = \psi$. This implies $h \circ \nu = 
\rho = g \circ \nu$ hence $g = h$. 
 \end{proof}

 \begin{propdefn}
 Given two homomorphisms 
 $$f \in \Hom{}_R(M.,  M'.) \text{ and } g \in \Hom{}_R 
(.N, .N').$$ 
 Then there is a unique homomorphism  
 $$f \otimes_R g \in \Hom (M \otimes_R N, M'\otimes_R N')$$ 
 such that $f \otimes_R g (m \otimes n)= f (m) \otimes 
g(n)$, i.e. the following diagram commutes 
 $$\bfig 
 \putmorphism(0, 500)(1, 0)[M \times N`M \tensor_R N`\tensor]{600}1a 
 \putmorphism(0, 0)(1, 0)[M' \times N'`M' \tensor_R N'`\tensor]{600}1b 
 \putmorphism(0, 500)(0, 1)[``f \times g]{500}1l 
 \putmorphism(500, 500)(0, 1)[``f \tensor_R g]{500}1r 
 \efig$$ 
 \end{propdefn}

 \begin{proof}
 $\otimes \circ (f \times g)$ is balanced.
 \end{proof}

 \begin{notatn}
 We often write $f \tensor_R N := f \tensor_R 1_N$ and $M 
\tensor_R g := 1_M \tensor_R g$.

 We have the following rule 
of computation:
 $$f \tensor_R g = (f \tensor_R N') \circ (M \tensor_R g) = 
(M' \tensor_R g) \circ (f \tensor_R N)$$  
since $f \times g = (f \times N') \circ (M \times g) = 
(M' \times g) \circ (f \times N)$. Observe that decomposable 
tensors $m \tensor n$ in $M \tensor_R N$ are written without 
the subscript $R$. The map $f \tensor_R g$ is (in general) 
not a decomposable tensor.
 \end{notatn}

 \section{Bisemimodules}

 \begin{defn}
 Let $R$, $S$ be semirings and let $M$ be a left $R$-semimodule and 
a right $S$-semimodule. $M$ is called an $R$-$S$-{\em bisemimo\-dule} 
if $(r m)s = r (m s)$. We define \\
 \phantom{XXXX}$\Hom_{R\x S} (.M., .N.): = \Hom_R (.M,.N) \cap \Hom_S (M.,N.).$
 \hfill $\Box$
 \end{defn} 

 \begin{rem}
 Let $M_S$ be a right $S$-semimodule and let $R \times M 
\to M$ be a map. $M$ is an $R$-$S$-bisemimodule if and only if 
 \begin{enumerate}
 \item $\forall r \in R: (M \ni m \mapsto r m \in M ) \in 
\Hom_S (M., M.) $, 
 \item $\forall r,r' \in R, m \in M: (r+r') m = r m + r'm$, 
 \item $\forall r, r' \in R, m \in M: (r r') m= r (r'm)$,
 \item $\forall m \in M: 1 m = m,$
 \item $\forall r \in R, m \in M: 0 \cdot m = 0 = r \cdot 0$.
 \end{enumerate}
 \end{rem}
 
 \begin{lma}
 \begin{enumerate}
 \item Let $ {}_RM_S$ and $ {}_SN_T$ be bisemimodules. Then\break $ {}_R(M 
\otimes_S N)_T$ is a bisemimodule by $r (m \otimes n) : 
= r m \otimes n$ and\break $(m \otimes n) t: = m \otimes n t$. 
 \item Let $ {}_RM_S$ and $ {}_RN_T$ be bisemimodules. Then 
${}_S \Hom_R(.M,.N)_T$ is a bisemimodule by $(s f)(m) : 
= f(ms)$ and $(ft)(m): = f(m)t$.
 \end{enumerate}
 \end{lma}

 \begin{proof}
  (1) Clearly we have that $r (m \otimes n) := (r \otimes_S \id) (m \otimes n) = r 
m \otimes n$ is a homomorphism. Then
 (2)-(5) hold. Thus $M \otimes_S N$ is a left $R$-semimodule. 
Similarly it is a right $T$-semimodule. Finally we have $r((m 
\tensor n)t) = r(m \tensor nt) = rm \tensor nt = (rm 
\tensor n)t = (r(m \tensor n))t$. 

 (2) is straightforward.
 \end {proof}

 \begin{cor} 
 Given bisemimodules $ {}_{R}M_S$, $ {}_{S}N_T$, $ {}_{R}M'_S$, 
$ {}_{S}N'_T$ and homomorphisms $f \in \Hom_{R\x S} 
(.M., .M'.)$ and $g \in \Hom_{S\x T} (.N., .N'.)$. Then we 
have $f \otimes_S g \in \Hom_{R\x T} (.M \otimes_S 
N., .M' \otimes_S N'.).$ 
 \end{cor} 

 \begin{proof}
 $(f \otimes_S g) (r m \otimes n t) = f (r m) \otimes g (n t 
) = r (f \otimes_S g ) (m \otimes n) t .$ 
 \end{proof}

 \begin{rem}
 (1) Let $M_R$, ${}_RN$, $M'_R$, and ${}_RN'$ be $R$-semimodules. 
Then the following is a homomorphism of commutative monoids:\\
 \phantom{X}$\mu: \Hom_R(M,M') \tensor_{\N_0} \Hom_R(N, N') \to \Hom(M \tensor_R N, M' \tensor_R 
N')$ \\
\phantom{XXXXXXXXXXXXXXXXXXX}$ f \tensor g \mapsto f \tensor_R g $. 

 (2) In general $\mu$ is not injective nor is $\mu$ 
surjective.

 (3) So $f \tensor g$ is a decomposable tensor whereas 
$f \tensor_R g$ is not a tensor, because it is not an element of a tensor product.
 \end{rem}

 \begin{thm} \label{coherencemaps}
 Let ${}_RM_S$, 
${}_SN_T$, and ${}_TP_U$ be bisemimodules. Then there are 
canonical isomorphisms of bisemimodules 
 \begin{enumerate}
 \item {\em Associativity Law:} $\alpha: (M \tensor_S N) 
\tensor_T P \iso M \tensor_S (N \tensor_T P)$. 
 \item {\em Law of the Left Unit:} $\lambda: R \tensor_R M 
\iso M$. 
 \item {\em Law of the Right Unit:} $\rho: M \tensor_S S 
\iso M$. 
 \item {\em Existence of Inner Hom-Functors:} Let 
${}_RM_T$, ${}_SN_T$, and ${}_SP_R$ be bisemimodules. 
Then $\Hom_T(M.,N.)$ is an $S$-$R$-bisemimodule and 
there are canonical isomorphisms of bisemimodules 
 $$\Hom_{S\x T}(.P \tensor_R M., .N.) \iso \Hom_{S\x 
R}(.P.,.\Hom_T(M.,N.).) \mbox{ and }$$ 
 $$\Hom_{S\x T}(.P \tensor_R M., .N.) \iso \Hom_{R\x 
T}(.M.,.\Hom_S(.P,.N).).$$ 
 \end{enumerate}
 \end{thm}

 \begin{proof} We only describe the corresponding 
homomorphisms. 

 (1) This map is the most difficult to define. We want a homomorphism 
that is based on 
 $$\alpha((m \tensor n) \tensor p) := m \tensor (n \tensor p).$$
 But our definition of the tensor product (Definition 
\ref{tensorproduct}) allows only balanced maps in 2 variables. 
So we have to define a $T$-balanced map
 $$\alpha': (M \tensor_S N) \times P \ni (\Sigma m_i \tensor n_i, p) \mapsto \Sigma m_i \tensor (n_i \tensor p) \in M \tensor_S (N \tensor_T P).$$
 To define this map we fix an element $p \in P$ and define a map 
 $$\beta_p: M \tensor_S N \ni \Sigma m_i \tensor n_i \mapsto \Sigma m_i \tensor (n_i \tensor p) \in M \tensor_S (N \tensor_T P).$$
 Again by definition of the tensor product we can only define {\em homomorphisms} with source $M \tensor_S N$. For that we have to show that 
 $$\beta'_p:M \times N \ni (m,n) \mapsto m \tensor (n \tensor p) \in  M \tensor_S (N \tensor_T P)$$
is an $S$-balanced map:
 $$\begin{array}{l}
 \beta'_p(m+m',n) = (m + m') \tensor (n \tensor p) = m \tensor (n \tensor p) + m' \tensor (n \tensor p) = \\
 \phantom{XXXXXXX}= \beta'_p(m,n) + \beta'_p(m',n),\\ 
 \beta'_p(m,n + n') = m \tensor ((n + n') \tensor p) = m \tensor (n \tensor p) + m \tensor (n' \tensor  p) =\\
 \phantom{XXXXXXX}= \beta'_p(m,n) + \beta'_p(m, n'),\\
 \beta'_p(ms,n) = ms \tensor (n \tensor p) = m \tensor s(n \tensor p) = m \tensor (sn \tensor p) =\\
 \phantom{XXXXXXX}= \beta'_p(m, sn),\\
 \beta'_p(0,n) = 0 \tensor (n \tensor p) = 0,\\
 \beta'_p(m, 0) = m \tensor (0 \tensor p) = m \tensor 0 = 0.
 \end{array}$$
 Thus $\beta'_p$ induces a unique homomorphism $\beta_p$ with
  $$\beta_p(m \tensor n) = m \tensor (n \tensor p) \mbox { and hence }\beta_p(\Sigma m_i \tensor n_i) = \Sigma m_i \tensor (n_i \tensor p)$$
 for each $p \in P$. Thus we have a map $\alpha'$ with 
 $$\alpha'(\Sigma m_i \tensor n_i,p) := \Sigma m_i \tensor (n_i \tensor p)$$
 It has the following properties (we omit summation over the elements $m_i \tensor n_i$):
 $$\begin{array}{ll}
 \alpha'(m \tensor n + m' \tensor n',p) &= \beta_p(m \tensor n + m' \tensor n') = \\
 &= \beta_p(m \tensor n) + \beta_p(m' \tensor n') = \\
 &= \alpha'(m \tensor n, p) + \alpha'(m' \tensor n', p),\\

 \alpha'(m \tensor n, p + p') &= \beta_{p+p'}(m \tensor n) = \\
 &= m \tensor (n \tensor (p + p')) = \\
 &= m \tensor (n \tensor p) + m \tensor (n \tensor p') =\\ 
 &= \beta_p(m \tensor n) + \beta_{p'}(m \tensor n) = \\
 &= \alpha'(m \tensor n, p) + \alpha'(m \tensor n, p'),\\

 \alpha'((m \tensor n)t, p) &= \alpha'(m \tensor nt, p) = m \tensor (nt \tensor p) = \\
 &= m \tensor (n \tensor tp) = \alpha'(m \tensor n, tp),\\

 \alpha'(0, p) &= \alpha'(0 \tensor 0,p) = 0 \tensor (0 \tensor p) = 0,\\
 \alpha'(m \tensor n, 0) &= m \tensor (n \tensor 0) = m \tensor 0 = 0.
 \end{array}$$
 So $\alpha'$ is $T$-balanced and generates a uniquely defined homomorphism
 $$\alpha: (M \tensor_S N) \tensor_T P \ni (m \tensor n) \tensor p \mapsto m \tensor (n \tensor p) \in M \tensor_S (N \tensor_T P).$$
 In an analogous way we get a homomorphism
 $$\gamma: M \tensor_S (N \tensor_T P) \ni m \tensor (n \tensor p) \mapsto (m \tensor n) \tensor p \in (M \tensor_S N) \tensor_T P$$
 and their composition $\alpha \circ \gamma$ and $\gamma \circ \alpha$ is the identity, thus $\alpha$ is an isomorphism.

 (2) Define $\lambda: R \tensor_R M \to M$ by $\lambda(r 
\tensor m) := r m$. 

 (3) Define $\rho : M \tensor_S S \to M$ by $\rho(m \tensor 
s) := m s$. 

 (4) For $f: P \tensor_R M \to N$ define $\varphi(f): P\to 
\Hom_T(M,N)$ by $\varphi(f)(p)(m) := f(p \tensor m)$ and 
$\varphi'(f): M \to \Hom_S(P,N)$ by $\varphi'(f)(m)(p)\break := f(p \tensor 
m)$. Conversely for $g: P \to \Hom_T(M, N)$ define $\psi(g)(p \tensor m) := g(p)(m)$. Then 
it is easy to verify that $\varphi$ and $\psi$ are inverse isomorphisms of each other.
 \end{proof}

 \begin{rem}\ \phantom{XXX}\\
 (1) The following diagrams ({\em coherence diagrams 
or constraints}) of ${\N_0}$-semimodules commute:
 $$\bfig
 \putmorphism(0, 800)(0, -1)[((A \tensor B) \tensor C) 
\tensor D`(A \tensor (B \tensor C)) \tensor D`\sst 
\alpha(A,B,C) \tensor 1]{400}1l 
 \putmorphism(0, 400)(0, -1)[\phantom{(A \tensor (B 
\tensor C)) \tensor D}`A\tensor ((B \tensor C) \tensor 
D)`\sst \alpha(A,B \tensor C, D)]{400}1l 
 \putmorphism(0, 800)(1, 0)[\phantom{(A \tensor B) \tensor C) 
\tensor D}`(A \tensor B) \tensor (C \tensor 
D)`\sst \alpha(A \tensor B,C,D)]{1400}1a 
 \putmorphism(0, 0)(1, 0)[\phantom{A\tensor ((B \tensor C) \tensor 
D)}`\phantom{A \tensor (B \tensor (C \tensor D))}`\sst 1 \tensor \alpha(B,C,D)]{1400}1a 
 \putmorphism(1400, 800)(0, -1)[(A \tensor B) \tensor (C \tensor 
D)`A \tensor (B \tensor (C \tensor D))`\sst \alpha(A,B, C 
\tensor D)]{800}1r 
 \efig$$
 $$\bfig
 \putmorphism(0, 400)(1, 0)[(A \tensor {\N_0}) \tensor B `A 
\tensor ({\N_0} \tensor B)`\sst \alpha(A,{\N_0},B)]{1200}1a 
 \putmorphism(0, 400)(3, -2)[`A \tensor B`\sst \rho(A) 
\tensor 1]{600}1l 
 \putmorphism(1200, 400)(-3, -2)[``\sst 1 \tensor 
\lambda(B)]{600}1r 
 \efig$$

 (2) Let $\tau(A,B): A \tensor B \to B \tensor A$ be defined by  
$\tau(A,B): a \tensor b \mapsto b \tensor a$. Then 
 $$\bfig
 \putmorphism(0, 500)(1, 0)[(A \tensor B) \tensor C`(B 
\tensor A) \tensor C`\sst \tau(X,B) \tensor 1 ]{1000}1a 
 \putmorphism(1000, 500)(1, 0)[\phantom{(B \tensor A) 
\tensor C}`B \tensor (A \tensor C)`\sst \alpha]{1000}1a 
 \putmorphism(0, 500)(0, -1)[``\sst \alpha]{500}1l
 \putmorphism(2000, 500)(0, -1)[``\sst 1 \tensor 
\tau(A,C)]{500}1r 
 \putmorphism(0, 0)(1, 0)[A \tensor (B \tensor C)`(B 
\tensor C) \tensor A`\sst \tau(A,B \tensor C)]{1000}1a 
 \putmorphism(1000, 0)(1, 0)[\phantom{(B \tensor C) \tensor 
A}`B \tensor (C \tensor A)`\sst \alpha]{1000}1a 
 \efig$$
commutes for all ${\N_0}$-semimodules $A, B, C$ and  
 $$\tau(B,A)\tau(A,B) = \text{id}_{A \tensor B}$$ for all 
${\N_0}$-semimodules $A$ and $B$.

 Let $f: A \to A'$ and $g: B \to 
B'$ be ${\N_0}$-semimodule homomorphisms. Then 
 $$\bfig
 \putmorphism(0, 500)(1, 0)[A \tensor B `B \tensor A`\sst 
\tau(A,B)]{700}1a 
 \putmorphism(0, 500)(0, -1)[``\sst f \tensor g]{500}1l 
 \putmorphism(700, 500)(0, -1)[``\sst g \tensor f]{500}1r 
 \putmorphism(0, 0)(1, 0)[A' \tensor B' `B' \tensor A'`\sst 
\tau(A',B')]{700}1a 
 \efig$$
  commutes.

 (3) There are examples of a commutative semiring $K$ and 
$K$-semi\-modules $M$, $N \in K\x\Mod\x K$ such that 
$M \tensor_K N \not\iso N \tensor_K M$.

\end{rem}

 \begin{prop} 
 Let $(RX, \iota)$ be a free $R$-semimodule and ${}_SM_R$ be a 
bisemimodule. Then every element $u \in M 
\otimes_R RX$ has a unique representation $u= \Sigma_{x \in X} 
m_x \otimes x$. 
 \end{prop}

 \begin {proof}
 By \ref{freerules} $\Sigma_{x \in X} r_x x$ is the 
general element of $RX$. Hence we have $u = \Sigma m_i 
\otimes \alpha_i = \Sigma m_i \otimes \Sigma r_{x, i}  x 
= \Sigma_i \Sigma_x m_i r_{x, i} \otimes  x = \Sigma_x 
(\Sigma_i m_i r_{x,i}) \otimes  x$.

 To show the uniqueness let $\Sigma_{y \in X} m_y \otimes  y 
= \Sigma_{y \in X} m'_y \otimes  y $. 
Let $x \in X$ and $f_x: RX \to R$ be defined by  
$f_x (\iota(y)) = f_x(y) := \delta_{xy}$. Then 
$ m_x \otimes 1 = 
\Sigma m_y \otimes f_x ( y) = 
(1_M \otimes_Rf_x)(\Sigma m_y \otimes  y) = 
(1_M \otimes_Rf_x)(\Sigma m'_y \otimes  y) = 
\Sigma m'_y \otimes f_x ( y) = 
m'_x \otimes 1$ for all $x \in X$. Now let 
 $$ \qtriangle[ M \times R ` M \tensor_R R ` M; \tensor ` 
\hbox{mult}` \rho ] $$
 be given. Then  $m_x  = m_x \cdot 1  = \rho (m_x \otimes 1) 
 = \rho (m'_x \otimes 1)  = m'_x \cdot 1  = m'_x $ 
hence we have uniqueness. From \ref{coherencemaps} (2) 
we know that $\rho$ is an isomorphism. 
 \end {proof}

 \begin{cor}
 Let ${}_SM_R$, ${}_RN$ be (bi-)semimodules. Let $M$ be a free  
$S$-semimodule over $Y$, and $N$ be a free $R$-semimodule over $X$. 
Then $M \otimes_R N$ is a free $S$-semimodule over $Y \times 
X$. 
 \end{cor}

 \begin {proof}
 Consider the diagram
$$\bfig 
 \putmorphism(0, 500)(1, 0)[Y \times X`M \times N`\iota_Y 
\times \iota_X]{1000}1a
 \putmorphism(1000, 500)(1, 0)[\phantom{M \times N}`{}_SM 
\tensor_RN`\tensor]{500}1a
 \putmorphism(0, 500)(3, -1)[`{}_SU`f]{1500}1l
 \putmorphism(1000, 500)(1, -1)[``g]{500}1l
 \putmorphism(1500, 500)(0, -1)[``h]{500}1r
 \efig$$

 Let $f$ be an arbitrary map. For all $x \in X$ we define homomorphisms  
$g(\x,x) \in \Hom_S (.M,.U)$ by the commutative diagram
 $$ \qtriangle[Y `{}_SM `{}_SU ;\iota_Y `f(\x,x) `g(\x,x) ] 
$$ 
 Let $\widetilde g \in \Hom_R (.N,.\Hom_S (.M_R, .U))$ be 
defined by 
 $$ \qtriangle[X `{}_RN `{}_R\Hom_S(.M_R,.U) ;\iota_x 
`g(\x,\x) `\widetilde g ] $$ 
 with $ x \mapsto g (\x,x)$. Then we define $g(m,n):= 
\widetilde g (n) (m) = :h (m \otimes n)$. Observe that $g$ is 
additive in $m$ and in $n$ (because $\widetilde g$ is 
additive in $m$ and in $n$), and $g$ is $R$-balanced, 
because $g (m r , n) = \widetilde g (n) (m r)= (r 
\widetilde g (n)) (m) = \widetilde g (rn) (m) = g (m,r n)$. 
Obviously $g (y,x) = f (y,x)$, hence $h \circ \otimes \circ 
\iota_Y \times \iota_X = f$. Furthermore we have $h (s m 
\otimes n) = \widetilde g (n) (s m ) = s (\widetilde g (n) 
(m)) =s h (m \otimes n)$, hence $h$ is an $S$-semimodule 
homomorphism. 

 Let $k$ be an $S$-semimodule homomorphism satisfying $k \circ 
\otimes \circ \iota_Y \times \iota_X = f$, then $k \circ 
\otimes (\x, x) = g (\x, x)$, since $k \circ \otimes$ is 
$S$-linear in the first argument. Thus $k \circ \otimes 
(m,n) = \widetilde g (n) (m) = h (m \otimes n)$, and hence 
$h = k$. 
 \end {proof}

 \end{document}